\documentclass[11pt]{amsart}
\usepackage{amsthm}

\usepackage[all]{xy}
\usepackage{amssymb}
\usepackage{enumerate}
\usepackage{mathrsfs}
\usepackage{epsfig}
\usepackage{graphicx}
\usepackage{subfig}
\usepackage{float}
\usepackage{epigraph}

\numberwithin{equation}{section}

\newtheorem{thm}{Theorem}[section]

\newtheorem{lem}[thm]{Lemma}
\newtheorem{prop}[thm]{Proposition}
\newtheorem{rem}{Remark}[section]
\newtheorem{example}[thm]{Example}

\newtheorem{defin}[thm]{Definition}

\newcommand{\eq}[1]{(\ref{#1})}

\renewcommand{\Re}{\operatorname{\rm Re}}
\renewcommand{\Im}{\operatorname{\rm Im}}

\newcommand{\beqast}{\begin{eqnarray*}}
\newcommand{\eqast}{\end{eqnarray*}}
\newcommand{\beqa}{\begin{eqnarray}}
\newcommand{\eqa}{\end{eqnarray}}

\newcommand{\bbe}{\begin{equation}}
\newcommand{\ee}{\end{equation}}

\renewcommand{\Re}{\operatorname{\rm Re}}
\renewcommand{\Im}{\operatorname{\rm Im}}

\newcommand{\bC}{{\mathbb C}}
\newcommand{\bE}{{\mathbb E}}
\newcommand{\bN}{{\mathbb N}}

\newcommand{\bQ}{{\mathbb Q}}

\newcommand{\bR}{{\mathbb R}}

\newcommand{\bZ}{{\mathbb Z}}

\newcommand{\cF}{{\mathcal F}}

\newcommand{\cE}{{\mathcal E}}
\newcommand{\cG}{{\mathcal G}}

\newcommand{\cL}{{\mathcal L}}

\newcommand{\cC}{{\mathcal C}}

\newcommand{\barX}{{\bar X}}
\newcommand{\uX}{{\underline X}}

\newcommand{\cEq}{{\mathcal E_q}}
\newcommand{\cEpq}{{\mathcal E^+_q}}
\newcommand{\cEmq}{{\mathcal E^-_q}}

\newcommand{\phipq}{{\phi^+_q}}
\newcommand{\phimq}{{\phi^-_q}}

\newcommand{\tV}{{\tilde V}}
\newcommand{\tW}{{\tilde W}}

\newcommand{\Om}{{\Omega}}

\newcommand{\be}{\beta}
\newcommand{\De}{\Delta}
\newcommand{\de}{\delta}
\newcommand{\eps}{\epsilon}

\newcommand{\lp}{\lambda_+}
\newcommand{\lm}{\lambda_-}
\newcommand{\La}{\Lambda}
\newcommand{\mum}{\mu_-}
\newcommand{\mup}{\mu_+}
\newcommand{\mumpr}{\mu'_-}
\newcommand{\muppr}{\mu'_+}

\newcommand{\num}{\nu_-}
\newcommand{\nup}{\nu_+}

\newcommand{\sg}{\sigma}

\newcommand{\om}{\omega}
\newcommand{\omm}{\om_-}
\newcommand{\omp}{\om_+}

\newcommand{\ze}{\zeta}

\newcommand{\ga}{\gamma}
\newcommand{\gap}{\gamma_+}
\newcommand{\gam}{\gamma_-}
\newcommand{\gappr}{\gamma'_+}
\newcommand{\gampr}{\gamma'_-}

\newcommand{\Ga}{\Gamma}

\newcommand{\barnu}{\bar\nu}

\newcommand{\bfo}{{\bf 1}}

\begin{document}

\title[Evaluation of expectations of functions of a L\'evy process and its extremum]
{Efficient evaluation  of expectations of functions of a L\'evy process and its extremum}
\author[
Svetlana Boyarchenko and
Sergei Levendorski\u{i}]
{
Svetlana Boyarchenko and
Sergei Levendorski\u{i}}

\begin{abstract}
We prove a simple general formula for the expectation of a function of a L\'evy process and its running extremum.
Under additional conditions, we derive analytical formulas using the Fourier/Laplace inversion and Wiener-Hopf factorization,
and discuss efficient numerical methods for realization of these formulas. As applications, the cumulative probability distribution
function of the process and its running supremum and the price of the option to exchange the supremum of a stock price for a power of 
the price
are calculated. The most efficient numerical methods use the sinh-acceleration technique and simplified trapezoid rule. The program in Matlab running on a Mac with moderate characteristics achieves the precision E-7 and better in several milliseconds, and E-14 - 
in a fraction of a second.

\end{abstract}

\thanks{
\emph{S.B.:} Department of Economics, The
University of Texas at Austin, 2225 Speedway Stop C3100, Austin,
TX 78712--0301, {\tt sboyarch@utexas.edu} \\
\emph{S.L.:}
Calico Science Consulting. Austin, TX.
 Email address: {\tt
levendorskii@gmail.com}}

\maketitle



\tableofcontents

\section{Introduction}\label{s:intro}
There exists a large body of literature devoted
to 
calculation of the expectation of a function of a L\'evy process and its running extremum, and related optimal stopping problems,
standard examples being  barrier and American options, and lookback options with barrier and/or American features. 
The general formulas for single barrier options with continuous monitoring were derived in \cite{KoBoL,barrier-RLPE,NG-MBS}
using the operator form of the Wiener-Hopf factorziation \cite{eskin}, under certain regularity conditions on the characteristic exponent.
In \cite{single}, the same formulas were proved for any L\'evy process. The first contribution of the paper is a similar simple general formula for the expectation of a function of a L\'evy process and its running extremum, evaluated at a deterministic time $T>0$. 

The pricing formulas are in terms of Laplace-Fourier inversion in dimensions 2 (first touch digitals and no-touch options),
 3 (barrier puts and calls, and joint probability distributions of a L\'evy process and its extremum), and 4 (more general options with lookback and barrier features). Hence, even marginally  accurate realizations of these formulas are far from trivial unless the characteristic exponent $\psi$ of the process is a rational function, hence, the Wiener-Hopf factors are rational functions as well. The factors are especially simple
 in the Double exponential jump-diffusion model (DEJD model) used in \cite{kou,KW1} and its generalization: Hyper-exponential jump-diffusion model (HEJD model) constructed independently in \cite{lipton-columbia,lipton-risk} (see also \cite{LiptonSelection}) and
\cite{amer-put-levy-maphysto,amer-put-levy}. In \cite{lipton-columbia,lipton-risk}, an explicit pricing formula for the joint distribution of
the L\'evy process and its extremum was derived using the Gaver-Stehfest algorithm (GS algorithm); the formula can be used to price options with barrier-lookback features. Later, a variation of the same technique was used in structural default models \cite{lipton-sepp}. In 
\cite{amer-put-levy-maphysto,amer-put-levy}, American options with finite time horizon are priced using the maturity randomization technique (Carr's randomization). An evident simplification of the latter method can be applied to barrier options
(see \cite{MSdouble}, where double-barrier options in regime-switching models are priced): the early exercise boundary is fixed and it is unnecessary to fund
an approximation to the boundary at each step of backward induction. In both cases (GS-algorithm and Carr's randomization), the main block is the evaluation of
the perpetual options. If the GS-algorithm is used, it may be necessary to use high precision arithmetics because the weights are very large (see, e.g., examples
in \cite{paraLaplace}. If the GS-algorithm can be used with double precision arithmetic, then, typically, the CPU time is smaller than
if Carr's randomization is applied. 

In the case of more general L\'evy processes, efficient calculations are much more difficult because the option price is very irregular at
the barrier and maturity. See the asymptotic analysis in \cite{NG-MBS,early-exercise,BIL,AsAndersenLipton}. 
The irregular behavior makes it difficult to evaluate the prices of perpetual options sufficiently accurately so that the GS-algorithm
or Carr's randomization can produce good results. Certain additional tricks \cite{single,BLdouble} can be used to do relatively accurate calculations in the state space but calculations in the dual space \cite{paraLaplace,paired,UltraFast} are significantly more efficient.  
 A simple very efficient algorithm derived in the paper is more efficient than the algorithms in
  the papers above; the algorithm is a more efficient variation of the algorithm in \cite{Contrarian}.
  Once a general exact formula in terms of a sum of 
 1-3 dimensional integrals  is derived, good changes of variables allows one to evaluate the integrals with an almost machine precision and
 at a much smaller CPU cost than using any previously developed method; the error tolerance of the order of E-7 
 can be satisfied in milliseconds using Matlab and Mac with moderate characteristics. The algorithm is short and involves a handful of vector operations and multiplication by  matrices of a moderate size at 3 places of the algorithm. We explain that the choice of an approximately optimal parameters of the numerical scheme simplifies significantly if the process is a Stieltjes-L\'evy process (SL-process). This class is defined in \cite{EfficientAmenable}, where it is shown that all popular classes of L\'evy processes bar the Merton model and Meixner processes are SL processes. For the Merton model and Meixner processes, the computational cost can be several times higher.   In the accompanying papers \cite{EfficientDiscExtremum,EfficientZsufficient,EfficientStableLevyExtremum,EfficientDoubleBarrier}, the method of the paper is modified and applied
  to options with discrete monitoring, stable L\'evy processes and double-barrier options. Note that the method of the present paper and its
  analog for the options with discrete monitoring are more efficient than the other methods available in the literature - see, e.g.,
  \cite{GSh,AAP,AKP,beta,KudrLev09,KudrLev11,BIL,HaislipKaishev14,FusaiGermanoMarazzina,kirkbyJCompFinance18,Linetsky08,feng-linetsky09,LiLinetsky2015} and the bibliographies therein.
  
 Let $X$ be a one-dimensional L\'evy process on the filtered probability space $(\Om, \cF, \{\cF_t\}_{t\ge 0}, \bQ)$
satisfying the usual conditions, and let $\bE$ be the expectation operator under $\bQ$. 
Let  $\barX_t=\sup_{0\le s\le t}X_s$ and $\uX_t=\inf_{0\le s\le t}X_s$ be the supremum
and infimum processes (defined path-wise, a.s.); $X_0=\barX_0=\uX_0=0$. For a measurable function $f$, consider 
$
V(f;T;x_1,x_2)=\bE[f(x_1+X_T, \max\{x_2, x_1+\barX_T\})],
$
where $T>0$ and $x_1\le x_2$ are real.
 In Section \ref{ss:main_theorems},
we derive simple explicit formulas for the Laplace transform $\tV(f;q;x_1,x_2)$ of $V(f;T;x_1,x_2)$ using the operator form of the Wiener-Hopf factorization technique
  \cite{BLSIAM02,NG-MBS,barrier-RLPE,EPV,IDUU,single}. Basic facts of the Wiener-Hopf factorization technique in the form used in the paper
and definitions of general classes of L\'evy processes amenable to efficient calculations are collected in Section \ref{s:prelim}.
The formulas are in terms of the (normalized) expected present value operators $\cEq$, $\cEpq$ and $\cEmq$ defined by 
$\cEq u(x)=\bE[u(x+X_{T_q})],$
$\cEpq u(x)=\bE[u(x+\barX_{T_q})]$, $\cEmq u(x)=\bE[u(x+\uX_{T_q})]$, where $q>0$ and $T_q$ is an exponentially distributed random variable of mean $1/q$ independent of $X$. 
In the case of bounded functions (Theorem \ref{thm:X_barX_exp}), the formulas are proved for any L\'evy process, stable ones including; in the case of functions of exponential growth (Theorem \ref{thm:X_barX_exp_2}), the tail(s) of the L\'evy 
density must decay exponentially. A special case $x_1=x_2=0$ appeared earlier in the working paper
\cite{KudrLev11}; the version formulated and proved in the present paper is more efficient for applications. Theorems
\ref{thm:X_barX_exp} and \ref{thm:X_barX_exp_2} generalize  formulas for  $\bE[f(x_1+X_{T_q}, \min\{x_2,\uX_{T_q}\})]$ and
$\bE[f(x_1+X_{T_q}, \max\{x_2,\barX_{T_q}\})]$ derived in  \cite{BLSIAM02,NG-MBS,barrier-RLPE,EPV,IDUU,single} for the payoff functions of the form
$f(x_1,x_2)=g(x_1)\bfo_{(h,+\infty)}(x_2)$ and $f(x_1,x_2)=g(x_1)\bfo_{(-\infty,h)}(x_2)$, respectively.  
In Section \ref{s:evalFT}, we use the Fourier transform and the equalities $\cE^\pm_q e^{ix\xi}=\phi^\pm_q(\xi) e^{ix\xi}$, where $\phi^\pm_q(\xi)$ are the Wiener-Hopf factors,
 to realize the general formula derived in Section \ref{exp_Levy_extremum} as a sum of integrals.   As applications of the general theorems, in Section \ref{s:two_examples}, we derive
explicit formulas for the cumulative distribution function (cpdf) of the L\'evy process and its supremum,
and for the option to exchange  $e^{\barX_T}$ for the power $e^{\be X_T}$. 
In Section \ref{s:numer}, 
we demonstrate how the sinh-acceleration technique used in \cite{SINHregular} 
to price European options and applied in \cite{Contrarian,ConfAccelerationStable,BSINH} to pricing barrier options, evaluation of special functions and the coefficients in BPROJ method respectively can be applied to greatly decrease the sizes of grids and  CPU time needed to satisfy the desired error tolerance. This feature makes the method of the paper more efficient than  methods that use the fast inverse Fourier transform, fast convolution or fast Hilbert transform.
The changes of variables must be in a certain agreement as in \cite{paraLaplace,paired}, where a less efficient
family of fractional-parabolic deformations was used.
Note that Talbot's deformation \cite{Talbot79} cannot be applied if the conformal deformations technique is applied to the integrals with respect to the other dual variables.
 In Section \ref{s:concl}, we summarize the results of the paper and outline several extensions of the method of the paper. 
We relegate to Appendices  technical details,
and the outline of other methods that are used to price options with barrier/lookback features. Figures and one of the  tables are in Appendix
\ref{ss:figures}.

\section{Preliminaries}\label{s:prelim}

\subsection{Wiener-Hopf factorization} 
Lemma \ref{l:deep} and equalities \eq{eq:operWHF} and \eq{whf0} below are three equivalent forms of the Wiener-Hopf factorization for L\'evy processes. Eq. \eq{whf0}
and \eq{eq:operWHF} are special cases of the Wiener-Hopf factorization in complex analysis and the general theory of 
boundary problems for pseudo-differential operators (pdo), where more general classes of functions and operators appear
(see, e.g., \cite{eskin}).

In probability, the version \eq{whf0} was obtained  (see \cite{sato} for references) before Lemma \ref{l:deep}; the version 
\eq{eq:operWHF} was proved in   \cite{BLSIAM02,NG-MBS,barrier-RLPE,EPV,IDUU} under additional regularity conditions on the process, and in \cite{single}, for any L\'evy process.

\begin{lem}\label{l:deep} (\cite[Lemma 2.1]{greenwood-pitman}, and
\cite[p.81]{RW})
Let $X$ and $T_q$ be as above.  Then
\begin{enumerate}[(a)]
\item the random variables $\barX_{T_q}$ and
$X_{T_q}-\barX_{T_q}$ are independent; and

\item the random variables $\uX_{T_q}$ and
$X_{T_q}-\barX_{T_q}$ are identical in law.
\end{enumerate}
\end{lem}
By symmetry, the statements (a), (b) are valid with $\barX$ and $\uX$ interchanged.

Two basic forms of
  the Wiener-Hopf factorization (both immediate from
 Lemma \ref{l:deep}) are   
 \bbe\label{eq:operWHF}
 \cEq=\cEpq\cEmq=\cEmq\cEpq,
 \ee
 \bbe\label{whf0}
 \frac{q}{q+\psi(\xi)}=\phipq(\xi)\phimq(\xi).
 \ee
Evidently, the EPV-operators are bounded operators in $L_\infty(\bR)$. In exponential L\'evy models,
 payoff functions may increase exponentially, hence, we  consider the action of the EPV operators in $L_\infty(\bR; w)$, $L_\infty$- spaces with the weights 
 $w(x)=e^{\ga x}$, $\ga\in [\mum,\mup]$, and $w(x)=\min\{e^{\mum x},e^{\mup x}\}$, where $\mum\le0\le \mup, \mum<\mup$; the norm
 is defined by $\|u\|_{L_\infty(\bR;w)} =\|wu\|_{L_\infty(\bR)}$.

 Recall that a function $f$ is said to be analytic in the closure of an open set $U$ if $f$ is analytic in the interior of $U$ and continuous up to the boundary of $U$. We need the following straightforward result (see, e.g., \cite{NG-MBS,IDUU}).
 \begin{lem}\label{phipmq_anal_cont}
 Let  there exist  $\mum\le 0\le \mup$, $\mum<\mup$, such that  $\bE[e^{-\ga X_1}]<\infty$, $\forall\ \ga\in [\mum,\mup]$. 
 
 Then
 \begin{enumerate}[(i)]
 \item
 $\psi(\xi)$ admits analytic continuation to the strip $S_{[\mum,\mup]}:=\{\xi\in \bC\ |\
 \Im\xi\in [\mum,\mup]\}$; 
 \item 
 let $\sg>\max\{-\psi(i\mum), -\psi(i\mup)\}$. Then there exists $c>0$  s.t. $|q+\psi(\xi)|\ge c$ for $ q\ge \sg$ and $\xi\in S_{[\mum,\mup]}$;
 \item
   let  $q\ge \sg$. Then  $\phipq(\xi)$ (resp., $\phimq(\xi)$) admits analytic continuation to  $\{\Im\xi\ge \mum\}$
   (resp., $\{\Im\xi\le \mup\}$) given by    \beqa\label{analcontphipq}
   \phipq(\xi)&=&\frac{q}{(q+\psi(\xi))\phimq(\xi)}, \ \Im\xi\in [\mum,0],
   \\\label{analcontphimq}
   \phimq(\xi)&=&\frac{q}{(q+\psi(\xi))\phipq(\xi)}, \ \Im\xi\in [0,\mup];
   \eqa
  \item   $\phipq(\xi)$ (resp., $\phimq(\xi)$) is uniformly bounded on  $\{\Im\xi\ge \mum\}$
   (resp., $\{\Im\xi\le \mup\}$);
   \item
for any weight function of the form $w(x)=e^{\ga x}$, $\ga\in [\mum,\mup]$, and $w(x)=\min\{e^{\mum x},e^{\mup x}\},$
    operators $\cE^\pm_q$ are bounded in $L_\infty(\bR;w)$.
 
   \end{enumerate}
    \end{lem}
  We have $\cE^\pm_q e^{ix\xi}=\phi^\pm_q(\xi) e^{ix\xi}$. Hence, $\cE^\pm_q$ are pseudo-differential operators with symbols $\phi^\pm_q$,  which means that
 $\cE^\pm_qu(x)=\cF^{-1}_{\xi\to x}\phi^\pm_q(\xi)\cF_{x\to\xi}u(x)$ for sufficiently regular functions $u$.
 
 \subsection{General classes of L\'evy processes amenable to efficient calculations}\label{ss:gen_eff_Levy}
The conditions of Lemma \ref{phipmq_anal_cont} are satisfied for all popular classes of L\'evy processes
bar stable L\'evy processes. See \cite{NG-MBS,barrier-RLPE,BLSIAM02}, where the general class of Regular L\'evy processes of
exponential type (RLPE) is introduced. An additional property useful for development of efficient numerical methods
is a regular behavior of the characteristic exponent at infinity. In the definition below, we relax the conditions in \cite{NG-MBS,barrier-RLPE,BLSIAM02} allowing for non-exponential decay of one of the tails of the L\'evy density. Indeed, for calculations in the dual space,
it does not matter whether the strip of analyticity contains the real line or is adjacent to the real line.

For $\nu=0+$ (resp., $\nu=1+$), set $|\xi|^\nu=\ln|\xi|$ (resp., $|\xi|^\nu=|\xi|\ln|\xi|$), and introduce the following complete ordering in
the set $\{0+,1+\}\cup (0,2]$: the usual ordering in $(0,2]$; $\forall\ \nu>0, 0+<\nu$; $\forall\ \nu>1, 1<1+<\nu$.
We use    coni
 $\cC_{\gam,\gap}=\{e^{i\varphi}\rho\ |\ \rho> 0, \varphi\in (\gam,\gap)\cup (\pi-\gap,\pi-\gam)\}$, 
 $\cC_{\ga}=\{e^{i\varphi}\rho\ |\ \rho> 0, \varphi\in (-\ga,\ga)\}$, and the strip $S_{(\mum,\mup)}=\{\xi\ |\ \Im\xi\in (\mum,\mup)\}$.

 \begin{defin}\label{def:SINH_reg_proc_1D0}(\cite[Defin. 2.1]{EfficientAmenable})
 We say that $X$ is a SINH-regular L\'evy process  (on $\bR$) of   order
 $\nu$ and type $((\mum,\mup);\cC; \cC_+)$,
 iff
the following conditions are satisfied:
\begin{enumerate}[(i)]
\item
$\nu\in\{0+,1+\}\cup (0,2]$; $\mum<0\le \mup$ or $\mum\le 0<\mup$;
\item
$\cC=\cC_{\gam,\gap}, \cC_+=\cC_{\gampr,\gappr}$, where $\gam<0<\gap$, $\gam\le \gampr\le 0\le \gappr\le \gap$,
and $|\gampr|+\gappr>0$; 
\item
the characteristic exponent $\psi$ of $X$ can be represented in the form
\bbe\label{eq:reprpsi}
\psi(\xi)=-i\mu\xi+\psi^0(\xi),
\ee
where $\mu\in\bR$, and 
$\psi^0$ admits analytic continuation to $i(\mum,\mup)+ (\cC\cup\{0\})$;
\item
for any $\varphi\in (\gam,\gap)$, there exists $c_\infty(\varphi)\in \bC\setminus (-\infty,0]$ s.t.
\begin{equation}\label{asympsisRLPE}
\psi^0(\rho e^{i\varphi})\sim  c_\infty(\varphi)\rho^\nu, \quad \rho\to+\infty;
\end{equation}
\item
the function $(\gam,\gap)\ni \varphi\mapsto c_\infty(\varphi)\in \bC$ is continuous;
\item
for any $\varphi\in (\gampr, \gappr)$, $\Re c_\infty(\varphi)>0$.
\end{enumerate}
\end{defin}
\begin{example}\label{ex:KoBoL}{\rm  A generic process of Koponen's family was constructed in  \cite{genBS,KoBoL}  as a mixture of spectrally negative and positive pure jump processes, with the L\'evy measure
\begin{equation}\label{KBLmeqdifnu}
F(dx)=c_+e^{\lm x}x^{-\nu_+-1}\bfo_{(0,+\infty)}(x)dx+
 c_-e^{\lp x}|x|^{-\nu_--1}\bfo_{(-\infty,0)}(x)dx,
\end{equation}
where $c_\pm>0, \nu_\pm\in [0,2), \lm<0<\lp$. Starting with \cite{EfficientAmenable}, we allow for $c_+=0$ or $c_-=0$,
 $\lm=0<\lp$ and $\lm<0\le \lp$. This generalization is almost immaterial for evaluation of probability distributions and expectations because
 for efficient calculations, the first crucial property, namely, the existence of a strip of analyticity of the characteristic exponent, around or adjacent to the real line, holds if $\lm<\lp$ and $\lm\le 0\le \lp$. \footnote{The property does not hold
 if there is no such a strip (formally, $\lm=0=\lp$). The classical example are stable L\'evy processes.
 The conformal deformation technique  can be modified for this case as well \cite{ConfAccelerationStable}.}
 Furthermore, the Esscher transform allows one to reduce both cases $\lm=0<\lp$ and $\lm<0\le \lp$ to the case $\lm<0<\lp$.
 If $\nu_\pm\in (0,2), \nu_\pm\neq 1$,
\bbe\label{KBLnupnumneq01}
\psi^0(\xi)=c_+\Ga(-\nu_+)((-\lm)^{\nu_+}-(-\lm-i\xi)^{\nu_+})+c_-\Ga(-\nu_-)(\lp^{\nu_-}-(\lp+i\xi)^{\nu_-}).
\ee
 Note that a specialization
 $\nu_\pm=\nu\neq 1$, $c=c_\pm>0$, of KoBoL used in a series of numerical examples in \cite{genBS} was named CGMY model in \cite{CGMY} (and the labels were changed:
 letters $C,G,M,Y$ replace the parameters $c,\nu,\lm,\lp$ of KoBoL):
 \bbe\label{KBLnuneq01}
 \psi^0(\xi)= c\Ga(-\nu)[(-\lm)^{\nu}-(-\lm- i\xi)^\nu+\lp^\nu-(\lp+ i\xi)^\nu].
\ee
Evidently, $\psi^0$ given by \eq{KBLnuneq01} is analytic in $\bC\setminus i\bR$, and  $\forall\ \varphi\in (-\pi/2,\pi/2)$, \eq{asympsisRLPE} holds with
\bbe\label{ascofnupeqnumcc}
c_\infty(\varphi)=-2c\Ga(-\nu)\cos(\nu\pi/2)e^{i\nu\varphi}.
\ee
}
\end{example}
In \cite{EfficientAmenable}, we defined a class of Stieltjes-L\'evy processes (SL-processes). In order to save space, we do not reproduce the complete set of definitions. Essentially, $X$ is called a (signed) SL-process if $\psi$ is of the form
\bbe\label{eq:sSLrepr}
\psi(\xi)=(a^+_2\xi^2-ia^+_1\xi)ST(\cG^0_+)(-i\xi)+(a^-_2\xi^2+ia^-_1\xi)ST(\cG^0_-)(i\xi)+(\sg^2/2)\xi^2-i\mu\xi, 
\ee
where $ST(\cG)$ is the Stieltjes transform of a (signed) Stieltjes measure $\cG$,  $a^\pm_j\ge 0$, and $\sg^2\ge0$, $\mu\in\bR$.
We call a (signed) SL-process regular if it is SINH-regular. We proved in \cite{EfficientAmenable} that if $X$ is a (signed) SL-process then $\psi$ admits analytic continuation to the complex plane with two cuts along the imaginary axis, and
if $X$ is a SL-process, then, for any $q>0$, equation $q+\psi(\xi)=0$ has no solution on $\bC\setminus i\bR$.
We also proved that all popular classes of L\'evy processes bar the Merton model and Meixner processes are regular SL-processes, with $\ga_\pm=\pm \pi/2$;
the Merton model and Meixner processes are regular signed SL-processes, and $\ga_\pm=\pm \pi/4$.
 For lists of SINH-processes and SL-processes, with calculations of the order and type,
see \cite{EfficientAmenable}.

\subsection{Evaluation of  the Wiener-Hopf factors}\label{ss:expl_WHF}
      For  numerical realizations, we need the following explicit formulas for $\phi^\pm_q$ (see, e.g., \cite{NG-MBS,single,paraLaplace,paired,Contrarian}).
 \begin{lem}\label{phipmq_explicit1}
 Let $\mu_\pm$, $X$ and $q$ satisfy the conditions of Lemma \ref{phipmq_anal_cont}. Then
 \begin{enumerate}[(a)]
 \item
 for any $\omm\in (\mum,\mup)$ and $\xi\in\{\Im\xi>\omm\}$,
 \bbe\label{phip1}
 \phipq(\xi)=\exp\left[\frac{1}{2\pi i}\int_{\Im\eta =\omm}\frac{\xi\ln(1+\psi(\eta)/q)}{\eta(\xi-\eta)}d\eta\right];
 \ee
  \item
 for any $\omp\in (\mum,\mup)$ and $\xi\in\{\Im\xi<\omp\}$,
 \bbe\label{phim1}
 \phimq(\xi)=\exp\left[-\frac{1}{2\pi i}\int_{\Im\eta=\omp}\frac{\xi\ln(1+\psi(\eta)/q)}{\eta(\xi-\eta)}d\eta\right].
 \ee
\end{enumerate}
\end{lem}
The integrands above decay very slowly at infinity, hence,  fast and accurate numerical realizations are impossible unless additional tricks
are used. If $X$ is SINH-regular, the rate of decay can be greatly increased 
 using appropriate conformal deformations of the line of integration
 and the corresponding changes of variables. Assuming that in Definition \ref{def:SINH_reg_proc_1D0},  $\ga_\pm$ are not extremely small in absolute value (and, in the case of  regular SL-processes, $\ga_\pm=\pm \pi/2$ are not small), the most efficient change of variables
is the sinh-acceleration 
   \bbe\label{eq:sinh}
 \eta=\chi_{\om_1,b,\om}(y)=i\om_1+b\sinh(i\om+y), 
\ee
where $\om\in (-\pi/2,\pi/2)$, $\om_1\in \bR, b>0$.
Typically, the sinh-acceleration is the best choice even if $|\ga_\pm|$ are of the order of $10^{-5}$.  The parameters $\om_1,b,\om$ are chosen so that the contour $\cL_{\om_1,b,\om}:=\chi_{\om_1,b,\om}(\bR)\subset i(\mup,\mup)+\cC_{\gam,\gap}$ and, in the process of deformation, $\ln(1+\psi(\eta)/q)$ is a well-defined analytic function on a domain in $\bC$ or an appropriate Riemann surface.  \begin{lem}\label{lem:WHF-SINH}
Let $X$ be SINH-regular of type $((\mum,\mup), \cC_{\gam,\gap}, \cC_{\gampr,\gappr})$. 

Then  there exists $\sg>0$ s.t. for all $q>\sg$, 
\begin{enumerate}[(i)]
\item
$\phipq(\xi)$ admits analytic continuation to $i(\mum,+\infty)+i(\cC_{\pi/2-\gam}\cup\{0\})$. For any $\xi\in i(\mum,+\infty)+i(\cC_{\pi/2-\gam}\cup\{0\})$, and any contour 
$\cL^-_{\om_1,b,\om}\subset i(\mum,\mup)+(\cC_{\gam,\gap}\cup\{0\})$ lying below $\xi$,
\bbe\label{phipq_def}
\phipq(\xi)=\exp\left[\frac{1}{2\pi i}\int_{\cL^-_{\om_1,b,\om}}\frac{\xi \ln (1+\psi(\eta)/q)}{\eta(\xi-\eta)}d\eta\right];
\ee
\item
$\phimq(\xi)$ admits analytic continuation to $i(-\infty,\mup)-i(\cC_{\pi/2+\gap}\cup\{0\})$. For any $\xi\in i(\-\infty,\mup)-i(\cC_{\pi/2+\gap}\cup\{0\})$, and any contour 
$\cL^+_{\om_1,\om,b}\subset i(\mum,\mup)+(\cC_{\gam,\gap}\cup\{0\})$ lying above $\xi$,
\bbe\label{phimq_def}
\phimq(\xi)=\exp\left[-\frac{1}{2\pi i}\int_{\cL^+_{\om_1,\om,b}}\frac{\xi \ln (1+\psi(\eta)/q)}{\eta(\xi-\eta)}d\eta\right].
\ee
\end{enumerate}
\end{lem}
See Fig. \ref{fig:TwoGraphsandStrip} for an example of the curves $\cL^\pm_{\om_1,b,\om}$.
The integrals are efficiently evaluated making the change of variables $\xi=\chi_{\om_1,b,\om}(y)$ and applying the simplified trapezoid rule.

\begin{rem}\label{rem:SL-WHF}{\rm 
In the process of deformation, the expression $1+\psi(\xi)/q$ may not assume value zero. In order to avoid complications stemming from analytic continuation to an appropriate Riemann surface, it is advisable to ensure that $1+\psi(\xi)/q\not\in(-\infty,0]$.
 Thus, if $q>0$ - and only positive $q$'s are used in the Gaver-Stehfest method or GWR algorithm - and $X$ is a SL-process, 
any $\om\in (0,\pi/2)$ is admissible in \eq{phipq_def}, and any $\om\in (-\pi/2,0)$ is admissible in \eq{phimq_def}. 
If the sinh-acceleration is applied to the Bromwich integral, then additional conditions on $\om$ must be imposed. See 
Sect. \ref{ss:SINH-Bromwich}.
}
\end{rem}
\begin{rem}\label{rem:apmq}{\rm In the remaining part of the paper, we assume that the Wiener-Hopf factors $\phi^\pm_q(\xi), q>0,$ admit
the representations $\phi^\pm_q(\xi)=a^\pm_q+\phi^{\pm\pm}_q(\xi)$ and $\cE^\pm_q=a^\pm_qI+\cE^{\pm\pm}_q$,
where $a^\pm_q\ge 0$, and $\phi^{\pm\pm}_q(\xi)$ satisfy the bounds
\beqa\label{WHFdecayP}
 |\phi^{+,+}_q(\xi)|&\le & C_+(q)(1+|\xi|)^{-\nup},\ \Im \xi\ge \mum,\\
 \label{WHFdecayM}
 |\phi^{-,-}_q(\xi)|&\le & C_+(q)(1+|\xi|)^{-\num},\ \Im \xi\le \mup,
 \eqa
 where $\nu_\pm>0$ and $C_\pm(q)>0$ are independent of $\xi$. These conditions are satisfied for all popular classes of L\'evy processes
 bar the driftless Variance Gamma model.
See Sect. \ref{WHF_decomp} for details.
}
\end{rem}

\section{Expectations of functions of the L\'evy process and its extremum}\label{exp_Levy_extremum}

\subsection{Main  theorems}\label{ss:main_theorems}
Let $f$ be measurable and uniformly bounded on $U_+:=\{(x_1, x_2)\ |\ x_2\ge 0, x_1\le x_2\}$. 
Fix 
$(x_1,x_2)\in U_+$. The function
$\bR_+\ni T\mapsto V(f;T;x_1,x_2)$ is measurable and uniformly bounded, hence, $\tV(f;q;x_1,x_2)$, the Laplace transform of $V(f;T;x_1,x_2)$ w.r.t. $T$, is a well-defined analytic function of $q$ in the right half-plane.  Assuming  that $\tV(f;q;x_1,x_2)$ is sufficiently regular, $V(f;T;x_1,x_2)$
can be represented by the Bromwich integral
\bbe\label{tVBrom}
V(f;T;x_1,x_2)=\frac{1}{2\pi i}\int_{\Re q=\sg}e^{qT}\tV(f;q;x_1,x_2)\,dq,
\ee
where $\sg>0$ is arbitrary.  We derive an analytical representation for
\[
 \tV(f;q;x_1,x_2)=q^{-1}\bE[f(x_1+X_{T_q}, \max\{x_2, x_1+\barX_{T_q}\})],
 \]
 where $q>0$ and $T_q$ is an exponentially distributed random variable of mean $1/q$, independent of $X$,
 and prove that the resulting expression for $\tV(f;q;x_1,x_2)$ admits analytic continuation to the right half-plane.  One can impose additional general conditions on $X$ and $f$ which ensure that $\tV(f;q;x_1,x_2)$ is sufficiently regular
 so that \eq{tVBrom} holds. Such general conditions are either too messy or exclude some natural examples, for which the regularity can be established on the case-by-case basis. A standard trick which is used in \cite{NG-MBS,barrier-RLPE,BIL} is as follows. Firstly, \eq{tVBrom} holds in the sense of generalized functions. One integrates by parts
 in \eq{tVBrom}, and proves that the derivative  $\tV_q(f;q;x_1,x_2)$ is of class $L_1$ as a function of $q$. Hence, $V(f;T;x_1,x_2)$ 
equals the RHS of \eq{tVBrom} with $-T^{-1}\tV_q(f;q;x_1,x_2)$ in place of $\tV(f;q;x_1,x_2)$. After that, one proves that it is possible to integrate by parts back and obtain \eq{tVBrom} for 
$T>0$.  In examples that we consider, the integrands are of essentially the same form as in \cite{NG-MBS,barrier-RLPE,BIL} for 
barrier options, 
and enjoy all properties that are used
in \cite{NG-MBS,barrier-RLPE,BIL} to justify \eq{tVBrom}. 

In the theorem below, $I$ denotes the identity operator, $f_+$ is the extension of $f$ to $\bR^2$ by zero, and $\De$ is the diagonal map: $\De(x)=(x,x)$.

 \begin{thm}\label{thm:X_barX_exp}
 Let $X$ be a L\'evy process on $\bR$, $q>0$, and let $f:U_+\to \bR$ be a measurable and uniformly bounded  function
s.t.   $((\cEmq\otimes I)f)\circ \De:\bR\to\bR$ is measurable.
 Then 
 \begin{enumerate}[(i)]
  \item
 for any $x_1\le x_2$,
 \beqa\label{tVq0}
 q\tV(f;q;x_1,x_2)&=&((\cEq\otimes I)f_+)(x_1,x_2)+(\cEpq w(f;q,\cdot, x_2))(x_1),
 \eqa
 where 
 \bbe\label{eq:wqVtq}
 w(f;q,y,x_2)=\bfo_{[x_2,+\infty)}(y)(((\cEmq\otimes I)f_+)(y,y)-((\cEmq\otimes I)f_+)(y,x_2));
 \ee
  \item
   the RHS' of \eq{tVq0} and \eq{eq:wqVtq} admit analytic continuation w.r.t. $q$ to the right half-plane.
   \end{enumerate}
   
 \end{thm}
 \begin{proof}   By definition, part (a) of Lemma \ref{l:deep} amounts to the statement that the
probability distribution of the $\bR^2$-valued random variable
$(\barX_{T_q}, X_{T_q}-\barX_{T_q})$ is equal to the product (in the sense
of ``product measure'') of the distribution of $\barX_{T_q}$ and the
distribution of $X_{T_q}-\barX_{T_q}$. Applying Fubini's theorem and then part (b),
 we derive for $x_1\le x_2$
\beqast
&&\bE[f_+(x_1+X_{T_q}, \max\{x_2, x_1+\barX_{T_q}\})]\\
&=&\bE[f_+(x_1+X_{T_q}-\barX_{T_q}+\barX_{T_q}, \max\{x_2, x_1+\barX_{T_q}\})]
\\ &=&
\bE[((\cEmq\otimes I)f_+)(x_1+\barX_{T_q}, \max\{x_2, x_1+\barX_{T_q}\})]
\\
&=&\bE[((\cEmq\otimes I)f_+)(x_1+\barX_{T_q}, x_2)]
\\
&&+
\bE[\bfo_{x_1+\bar X_{T_q}\ge  x_2}(((\cEmq\otimes I)f_+)(x_1+\barX_{T_q}, x_1+\barX_{T_q})-((\cEmq\otimes I)f_+)(x_1+\barX_{T_q}, x_2))].
\eqast
Using \eq{eq:operWHF}, we write the first term on the rightmost side as $((\cEq\otimes I)f_+)(x_1,x_2)$; the second term is the second term on the RHS of \eq{tVq0}, which
finishes the proof of (i). As operators acting in the space of bounded measurable functions, $\cE^\pm_q$ 
admit analytic continuation w.r.t. $q$ to the right half-plane,  which proves (ii). \end{proof} 

\begin{rem}{\rm  The inverse Laplace transform of $q^{-1}(\cEq\otimes I)f_+(x_1,x_2)$ equals $\bE[f(x_1+X_T,x_2)]$,
and, therefore, can be easily calculated using the Fourier transform technique and sinh-acceleration \cite{SINHregular}. Essentially, we have the price of the
European option of maturity $T$, the riskless rate being 0, depending on $x_2$ as a parameter. Thus, the new element is the calculation of the second term on the RHS of \eq{tVq0}. We calculate both terms in the same manner in order to facilitate the explanation of various blocks of our method.}
 \end{rem}
 \begin{thm}\label{thm:X_barX_exp_2}
 Let   a L\'evy process $X$ on $\bR$, function $f:U_+\to \bR$ and real $q>0$ satisfy the following conditions \begin{enumerate}[(a)]
 \item
there exist $\mum\le 0\le \mup$ such that  $\forall$\ $\ga\in [\mum,\mup]$, $\bE[e^{-\ga X_1}]<\infty$ and $q+\psi(i\ga)>0$;
 \item
 $f$ is a measurable function admitting the bound
 \bbe\label{simple_bound_f_X_barX1}
 |f(x_1,x_2)|\le C(x_2)e^{-\mup x_1}, 
 \ee
 where $C(x_2)$ is independent of $x_1\le x_2$;
 \item
$((\cEmq\otimes I)f)\circ \De$ is a measurable admitting the bound
 \bbe\label{simple_bound_f_X_barX2}
|((\cEmq\otimes I)f)(x_1,x_1)| \le Ce^{-\mum x_1} , 
 \ee
 where $C$ is independent of $x_1\ge 0$.

 \end{enumerate}
 Then the statements (i)- (iii) of Theorem \ref{thm:X_barX_exp} hold.
 
 \end{thm}
 \begin{proof} It suffices to consider non-negative $f$. Define $f_n(x_1,x_2)=\min\{n, f(x_1,x_2)\}$, $n=1,2,\ldots$.
 By the dominated convergence theorem, $\tV(f_n;q;x_1,x_2)\uparrow \tV(f;q;x_1,x_2)$, a.s.,
 and since $f_n$ is bounded, \eq{tVq0} holds with $f_n$ in place of $f$.
 We rewrite \eq{tVq0} in the form
 \beqa\label{eq:X_barX_exp2}
 q \tV(f;q;x_1,x_2)&=&\bE[((\cEmq\otimes I)f_+)(x_1+\barX_{T_q}, \bar x_2)\bfo_{x_1+\bar X_{T_q}<\bar x_2}]\\\nonumber
&&+\bE[((\cEmq\otimes I)f_+)(x_1+\barX_{T_q}, x_1+\barX_{T_q})\bfo_{x_1+\bar X_{T_q}\ge \bar x_2}],
  \eqa
  and denote by $\tW_1(f_n;q;x_1,\bar x_2)+\tW_2(f_n;q;x_1, \bar x_2)$ the sum on
 the RHS of \eq{eq:X_barX_exp2} with $f_n$ in place of $f$. Fix $x_2$. On the strength of (a) and \eq{simple_bound_f_X_barX2},
$\tW_2(f_n;q;x_1, \bar x_2)$
admits a bound via $C\cEpq e^{-\mum x_1}=C\phipq(i\mum)e^{-\mum x_1}$, where $C$ is independent of $n$. On the strength of (a) and
\eq{simple_bound_f_X_barX1}, $\tW_2(f_n;q;x_1, \bar x_2)$
admits a bound via $C(\bar x_2)\cEpq\cEmq e^{-\mup x_1}=C_1(q,\bar x_2)e^{-\mup x_1}$, where $C_1(q,\bar x_2)$ is independent of $n$.
Operators $\cE^\pm_q$ being positive and bounded in $L_\infty$-spaces with weights $e^{\ga x}, \ga\in [-\mum,-\mum]$
(see Lemma \ref{phipmq_anal_cont}, (v)),
the limit of the RHSs of $\tW_1(f_n;q;x_1,x_2)+\tW_2(f_n;q;x_1,x_2)$ is finite  and equal  
 to the RHS of \eq{eq:X_barX_exp2}.
 
 \end{proof}
 
 \begin{rem}\label{rem:inf}{\rm 
 For functions of a L\'evy process and its running infimum, results are mirror reflections of the results for a L\'evy process and its supremum: change the direction of the real axis,
and flip the lower and upper half-plane and operators $\cE^\pm_q$.
 }\end{rem}
Let
$V(G; h; T; x)$ be the price of the barrier option with
the payoff $G(X_T)$ at maturity and no rebate if the barrier $h$ is crossed before or at time $n$; the rsikless rate is 0.
Applying Theorem   \ref{thm:X_barX_exp_2}, we obtain the formula for the price of the single barrier options, which is equivalent to
the formula derived
in \cite{KoBoL,NG-MBS,BLSIAM02,barrier-RLPE} for wide classes of L\'evy processes and generalized to all L\'evy processes in \cite{single}. The new version allows for more efficient numerical realizations.
\begin{thm}\label{thm:X_barX_exp_2_barr}
 Let   the L\'evy process $X$ on $\bR$ and $q>0$ satisfy condition (a) of Theorem \ref{thm:X_barX_exp_2}, and let
  $G$ be a measurable function admitting the bound
$ |G(x)|\le C(e^{-\mup x}+e^{-\mum x})$, where $C$ is independent of $x\in \bR$.
 Then, for $x<h$,
 \bbe\label{eq:price_barr}
 \tV(G;h;q,x)=q^{-1}(\cEq G)(x)-q^{-1}(\cEpq\bfo_{[h,+\infty)}\cEmq G)(x).
 \ee
 \end{thm}

 \subsection{Integral representation of 
  the Laplace transform of the value function}\label{s:evalFT} 
In this Section, we assume that  $q>0$. The RHS' of the formulas for the Wiener-Hopf factors and  formulas that we derive below admit analytic continuation w.r.t. $q$ so that the inverse Laplace transform can be applied. 
We assume that the representations  $\cE^{\pm}_q=a^\pm_qI+\cE^{\pm,\pm}_q$  (see Remark \ref{rem:apmq} and Lemma \ref{lem:atoms}) hold. This excludes the driftless Variance Gamma model which requires a separate treatment.  Using the equality
\[
w(f;q,x_1, x_2)=\bfo_{[x_2,+\infty)}(x_1)(((\cEmq\otimes I)f_+)(x_1,x_1)-((\cEmq\otimes I)f_+)(x_1,x_2))=0,\quad x_1\le x_2,
\]
we write the second term on
the RHS of \eq{tVq0} as 
\bbe\label{tVq02}
(\cEpq w(f;q,\cdot, x_2))(x_1)=(\cE^{++}_q w(f;q,\cdot, x_2))(x_1).
\ee
Similarly, we rewrite  \eq{eq:wqVtq} as
\bbe\label{eq:wqVtq2}
w(f;q, y, x_2)=a^-_q w_0(y,x_2)+w^-(f;q, y, x_2),
\ee
where $w_0(y,x_2)=\bfo_{[x_2,+\infty)}(y)(f_+(y,y)-f_+(y,x_2))$, and 
\bbe\label{eq:wqVtqm}
w^-(f;q, y, x_2)=\bfo_{[x_2,+\infty)}(y)(((\cE^{--}_q \otimes I)f_+)(y,y)-((\cE^{--}_q \otimes I)f_+)(y,x_2)).
\ee
Substituting \eq{eq:wqVtq2} into \eq{tVq02}, we obtain
\bbe\label{tVq03}
(\cEpq w(f;q,\cdot, x_2))(x_1)
=c^-_q(\cE^{++}_q\otimes I) w_0)(x_1,x_2)+ ((\cE^{++}_q\otimes I) w^-)(f;q,x_1, x_2).
\ee
In order to derive explicit integral representations for the terms on the RHS of \eq{tVq03}, we impose
 the following conditions, which can be relaxed: 
\begin{enumerate}[(a)]
\item
condition (a) of Theorem \ref{thm:X_barX_exp_2} is satisfied;
\item
 there exist $\mumpr, \muppr\in(\mum,\mup)$, $\mumpr<\muppr$ such that $f$ admits bounds 
 \beqa\label{simple_bound_f_X_barX1pr}
 |f(x_1,x_2)|&\le & C(x_2)e^{-\muppr x_1}, \ x_1\le x_2,
 \\\label{simple_bound_f_X_barX2pr}
|((\cEmq\otimes I)f_+)(x_1,x_1)|&\le& Ce^{-\mumpr x_1} , \ x_1\in\bR,
 \eqa
  where $C(x_2)$ and $C$ are independent of $x_1\le x_2$, and $x_1\in\bR$, respectively; 
  \item
for any $x_2$, there exists $C(x_2)>0$ such that 
 \beqa\label{boundf1}
 |\widehat {(f_+)}_1(\xi_1,x_2)|&\le&  C(x_2)(1+|\xi_1|)^{-1}, \quad \xi_1\in S_{[\muppr,\mup]},\\\label{boundfw}
 |\widehat {(w_0)}_1(\eta,x_2)|&\le&  C(x_2)(1+|\eta|)^{-1}, \quad \eta\in S_{[\mum,\mumpr]};
 \eqa
 \item
 there exists $C>0$  such that for $\xi_1\in S_{[\muppr,\mup]}$ and $\xi_2\in S_{[\mum,\mumpr]}$, 
 \bbe\label{boundf2}
 |\widehat {(f_+)}(\xi_1,\xi_2)|\le C(1+|\xi_1|)^{-1}(1+|\xi_2|)^{-1}.
 \ee
\end{enumerate}

\begin{thm}\label{thm:mainFT} Let conditions (a)-(d) hold and let the representations of the Wiener-Hopf factors in 
Remark \ref{rem:apmq} be valid.
Then, for any $\om,\om_1, \om_2$ and $\omm$ satisfying
\bbe\label{ompmom1om2}
\om, \om_1\in (\muppr,\mup),\
   \om_2\in (\mum,\mumpr),\ \omm\in (\mum,\om_1+\om_2),
 \ee 
 and $x_1\le x_2$, we have
 \beqa\label{eq:tVqmain}
\tV(f;q;x_1,x_2)&=&\frac{1}{2\pi }\int_{\Im\xi=\om}\frac{e^{ix_1\xi_1}}{q+\psi(\xi_1)}\widehat{(f_+)}(\xi_1,x_2)d\xi_1
 \\\nonumber
 &&+\frac{a^-_q}{2\pi q}\int_{\Im\eta=\omm} e^{ix_1\eta}\phi^{++}_q(\eta)\widehat {(w_0)}_1(\eta,x_2) d\eta
 \\\nonumber
&&+\frac{1}{2\pi q}\int_{\Im\eta=\omm} e^{i(x_1-x_2)\eta}\phi^{++}_q(\eta)\widehat{w^-_0}(f;q,\eta,x_2)d\eta,
\eqa
where $\widehat{w^-_0}(f;q,\eta,x_2)$ is given by
\beqa\label{eq:wqtVq4}
&&
\widehat{w^-_0}(f;q,\eta,x_2)
\\\nonumber
&=&\frac{1}{2\pi}\int_{\Im\xi_1=\om_1} d\xi_1 \, 
\frac{e^{ix_2\xi_1}}{i(\xi_1-\eta)}\phi^{--}_q(\xi_1)(\widehat{f_+})_1(\xi_1,x_2)\\\nonumber
&&+\frac{1}{(2\pi)^2}\int_{\Im\xi_1=\om_1} \int_{\Im\xi_2=\om_2}d\xi_1\,d\xi_2\, \frac{e^{ix_2(\xi_1+\xi_2)}}{i(\eta-\xi_1-\xi_2)}
\phi^{--}_q(\xi_1)(\widehat{f_+})(\xi_1,\xi_2).
\eqa
\end{thm}
\begin{proof} The first term on the RHS of \eq{tVq0} is $((\cEq\otimes I)f_+)(x_1,x_2)$, and the first term on the RHS of
\eq{eq:tVqmain} is
$q^{-1}((\cEq\otimes I)f_+)(x_1,x_2)$. Consider the second term on the RHS of \eq{tVq0}. We use \eq{tVq03}.
Since \eq{boundfw} holds and $\phi^{++}_q(\eta)=O(|\eta|^{-\nup})$ as $\eta\to\infty$ in the strip $S_{[\mum,\mup]}$,
where $\nup>0$, the integral
\bbe\label{eq:cEppw0}
((\cE^{++}_q\otimes I) w_0)(x_1,x_2)=\frac{1}{2\pi}\int_{\Im\eta=\omm} e^{ix_1\eta}\phi^{++}_q(\eta)\widehat {(w_0)}_1(\eta,x_2) d\eta
\ee
is absolutely convergent. It remains to consider $(\cE^{++}_q w^-(f;q,\cdot, x_2))(x_1)$.
If $\Im\eta=\om_-$,
\beqast
\widehat{w^-}(f;q,\eta,x_2)&=&-\int_{x_2}^{+\infty} dy\, e^{-iy\eta}\frac{1}{2\pi}\int_{\Im\xi_1=\om_1} d\xi_1 \, e^{i\xi_1 y}
\phi^{--}_q(\xi_1)(\widehat{f_+})_1(\xi_1,x_2)\\
&&+\int_{x_2}^{+\infty} dy\,e^{-iy\eta}\frac{1}{(2\pi)^2}\int_{\Im\xi_1=\om} \int_{\Im\xi_2=\om_2}d\xi_1\,d\xi_2\, e^{i(\xi_1+\xi_2) y}
\phi^{--}_q(\xi_1)(\widehat{f_+})(\xi_1,\xi_2).
\eqast
We apply Fubini's theorem to the first integral.   The integral $\int_{x_2}^{+\infty} dy\,e^{i(-\eta+\xi_1)y}=\frac{e^{ix_2(\xi_1-\eta)}}{i(\eta-\xi_1)}$ converges absolutely since $-\omm+\om_1>0$, and the repeated integral converges absolutely 
because $\phi^{--}_q(\xi)$ is uniformly bounded on the line of integration and \eq{boundf1} holds. Similarly, since 
$-\omm+\om_1+\om_2>0$, the integral $\int_{x_2}^{+\infty} dy\,e^{i(-\eta+\xi_1+\xi_2)y}=e^{ix_2(\xi_1+\xi_2-\eta)}/(i(\eta-\xi_1-\xi_2))$ converges absolutely. Since $\phi^{--}_q(\xi)=O(|\xi_1|^{-\num})$ as $\xi_1\to \infty$ along the line of integration, where $\num>0$,  \eq{boundf2} holds, and
 \bbe\label{bound_L1_1}
 \int_\bR \int_\bR d\xi_1\,d\xi_2\, (1+|\xi_1+\xi_2|)^{-1}(1+|\xi_1|)^{-1-\num}(1+|\xi_2|)^{-1}<\infty
 \ee
(see Sect. \ref{ss:proof_lemm_g} for the proof), the Fubini's theorem is applicable to the second integral as well. Thus,
 \bbe\label{eq:wqtVq3}
 \widehat{w^-}(f;q,\eta,x_2)=e^{-i\eta x_2}\widehat{w^-_0}(f;q,\eta,x_2),
 \ee
 where $\widehat{w^-_0}(f;q,\eta,x_2)$ is given by \eq{eq:wqtVq4}, and we obtain the triple integral\footnote{Recall that $\hat w^-_0$ is given by
the double integral \eq{eq:wqtVq4}.}
\bbe\label{eq:cEppwm}
(\cE^{++}_q w^-(\cdot,x_2))(x_1)=\frac{1}{2\pi}\int_{\Im\eta=\omm} e^{i(x_1-x_2)\eta}\phi^{++}_q(\eta)\widehat{w^-_0}(f;q,\eta,x_2)d\eta.
\ee
The integrand admits a bound via $Cg(\eta,\xi_1,\xi_2)$, where
\[
g(\eta,\xi_1,\xi_2)=(1+|\eta|)^{-\nup}(1+|\eta-\xi_1-\xi_2|)^{-1}(1+|\xi_1|)^{-1-\num}(1+|\xi_2|)^{-1}.
\]
Since 
\bbe\label{getaxi}
\int_{\bR^3}g(\eta,\xi_1,\xi_2)d\eta\,d\xi_1\,d\xi_2<\infty
\ee
(see Sect. \ref{ss:proof_lemm_g} for the proof), the triple integral on the the RHS of \eq{eq:cEppwm}  is absolutely convergent.
Substituting  \eq{tVq03}, \eq{eq:cEppw0} and \eq{eq:cEppwm} 
   into \eq{tVq0}, we obtain \eq{eq:tVqmain}.
\end{proof}
\begin{rem}\label{rem:directwm} {\rm In standard situations such as in the two examples that we consider below,
the function $y\mapsto h(y):= (\cE^{--}_q\otimes I)f_+(y,y)-(\cE^{--}_q\otimes I)f_+(y,x_2)$ is
a linear combination of exponential functions  (with the coefficients depending on $x_2$). Then $\widehat{w^-}(q;\eta,x_2)$ can be calculated directly, the double integral on the RHS of \eq{eq:wqtVq4} can be reduced to 1D integrals,
and the condition \eq{boundf2} replaced with the condition on $h$ similar to   \eq{boundfw}.
Analogous  simplifications are possible in more involved cases when $h$ is a piece-wise exponential polynomial in $y$.
}
\end{rem}

  \subsection{Two examples}\label{s:two_examples}
 \subsubsection{Example I. The joint cpdf of $X_T$ and $\barX_T$
  }\label{ss:jointpdf}
  For $a_1\le a_2$, and $x_1\le x_2$, set $f(x_1,x_2)=\bfo_{(-\infty,\min\{a_1,x_2\}]}(x_1)\bfo_{(-\infty,a_2]}(x_2)$ and consider 
\[
  V(f;T,x_1,x_2)=\bQ[x_1+X_T\le a_1, \max\{x_2,x_1+\barX_T\}\le a_2].
  \] 
 If $x_2>a_2$, then $V(f;T,x_1,x_2)=0$. Hence, we assume that
  $x_2\le a_2$.   
    \begin{thm}\label{thm:CPDF}  Let $q>0$, $a_1\le a_2, x_1\le x_2\le a_2$, 
and let $X$ satisfy conditions of Theorem \ref{thm:mainFT}.   Then, for any $\mum<\omm<0<\om_1<\mup$, 
     \beqa\label{tVq3}
&& \tV(f;q,x_1,x_2)\\\nonumber
 &=& \frac{1}{2\pi}\int_{\Im\xi_1=\om_1}\frac{e^{i(x_1-a_1)\xi_1}}{-i\xi_1(q+\psi(\xi_1))}d\xi_1
 \\\nonumber
 &&+\frac{1}{(2\pi)^2 q}\int_{\Im\eta=\omm}d\eta\,e^{i(x_1-a_2)\eta}\phi^{++}_q(\eta)
  \int_{\Im\xi_1=\om_1}d\xi_1\, \frac{e^{i\xi_1(a_2-a_1)}\phi^{--}_q(\xi_1)}{\xi_1(\xi_1-\eta)}.
\eqa 
    \end{thm}
    
  \begin{proof} We have $f_+(x_1,x_2)=\bfo_{(-\infty,a_1]}(x_1)\bfo_{(-\infty,a_2]}(x_2), $ therefore, for $x_2\le a_2$, 
    \beqast
    w_0(y,x_2)&=&\bfo_{[x_2,+\infty)}(y)\bfo_{(-\infty,a_1]}(y)(\bfo_{(-\infty,a_2]}(y)-\bfo_{(-\infty,a_2]}(x_2))\\
    &=&-\bfo_{[x_2,+\infty)}(y)\bfo_{(-\infty,a_1]}(y)\bfo_{(a_2,+\infty)}(y)=0,
    \eqast
    hence, the second term on the RHS of \eq{eq:tVqmain} is 0. Next, 
    \[
    \widehat{(f_+)}_1(\xi_1,x_2)=\bfo_{(-\infty,a_2]}(x_2)\int_{-\infty}^{a_1}e^{-ix_1\xi_1}d\xi_1=\bfo_{(-\infty,a_2]}(x_2)\frac{e^{-ia_1\xi_1}}{-i\xi_1}d\xi_1
    \]
    is well-defined in the upper half-plane, and satisfies the bound \eq{boundf1} in any strip $S_{[\muppr, \mup]}$, where $\muppr\in (0,\mup)$.   Hence, the first term on the RHS of   \eq{eq:tVqmain} becomes the first term on the RHS of \eq{tVq3}.
      It remains to evaluate the double integral on the RHS of \eq{eq:tVqmain}. As mentioned in Remark \ref{rem:directwm}, in the present case, it is simpler to directly evaluate $w^-$ and then $\widehat{w^-}$: for any $x_2\le a_2$, $\om_1\in (0,\mup)$ and any $\eta\in \{\Im\eta\in (\mum,\om_1)\}$,
    \beqa\nonumber
    w^-(q,y,x_2)&=&\bfo_{(x_2,+\infty)}(y)(\cE^{--}_q\bfo_{(-\infty,a_1]})(y)(\bfo_{(-\infty,a_2]}(y)-1)
    \\\nonumber&=&
   -\bfo_{[a_2,+\infty)}(y)(\cE^{--}_q\bfo_{(-\infty,a_1]})(y)\\\nonumber
      &=&-\bfo_{(a_2,+\infty)}(y)\frac{1}{2\pi}\int_{\Im \xi_1=\om_1}d\xi_1\, e^{i(y-a_1)\xi_1}\frac{\phi^{--}(\xi_1)}{-i\xi_1},
    \eqa
    \beqa\label{hwm}
    \widehat{w^-}(q,\eta,x_2)&=&-\int_{a_2}^{+\infty}e^{-iy\eta}\frac{1}{2\pi}\int_{\Im \xi_1=\om_1}d\xi_1\, e^{i(y-a_1)\xi_1}\frac{\phi^{--}(\xi_1)}{-i\xi_1}\\\nonumber
    &=& -\frac{e^{-ia_2\eta}}{2\pi}\int_{\Im \xi_1=\om_1}d\xi_1\, e^{i(a_2-a_1)\xi_1}\frac{\phi^{--}(\xi_1)}{i(\eta-\xi_1)(-i\xi_1)}.
    \eqa
    It is easy to see that both integrals are absolutely convergent.
    Substituting \eq{hwm} into the double integral on the RHS of \eq{eq:tVqmain}, we obtain \eq{tVq3}.
    \end{proof}
    \begin{rem}{\rm If $x_1>a_1$, then it advantageous to move the line of integration in the first integral on the RHS of
    \eq{tVq3} down, and, on crossing the simple pole, apply the residue theorem. In the result, the first term
    on the RHS  turns into
    \[ 
    \frac{1}{q}+\frac{1}{2\pi}\int_{\Im\eta=\omm}\frac{e^{i(x_1-a_1)\eta}}{-i\eta(q+\psi(\eta))}d\eta.
    \]
        }
   \end{rem}
      \begin{rem}\label{rem:cpdf_simpl}{\rm
    The first step of the proof of Theorem \ref{thm:CPDF} implies that we can replace $\phi^{--}_q$ in the double integral on the RHS
    of \eq{tVq3} with $\phimq$. From the computational point of view, if we make the conformal change of variables, this change does not lead to a significant increase in sizes of arrays necessary
    for accurate calculations, especially if $a_2-a_1>0$. The advantage is that it becomes unnecessary to evaluate $a^-_q$.
    Recall that the same $a^-_q$ appears for all $\xi_1$ in the formula $\phi^{--}_q(\xi_1)=\phimq(\xi_1)-a^-_q$,
    hence, it  is necessary to evaluate $a^-_q$ with a higher precision that $\phimq(\xi_1)$. At the same time, the integrand in the formula for
    $a^-_q$ decays slower at infinity than the integrand in the formula for $\phimq(\xi_1)$, hence, a significantly longer grid is needed to evaluate
    $\phimq(\xi_1)$ sufficiently accurately.
    }
    \end{rem}
     \begin{rem}\label{rem:no-touch}{\rm 
     Denote by $I_2(q;x_1,x_2)$ the double integral on the RHS of \eq{tVq3} multiplied by $q$.
 It follows from \eq{tVq02} that we can replace $\phi^{++}_q$ in the double integral 
     with $\phipq$. If $a_1<a_2$ and the conformal deformations are used, then this replacement causes no serious computational problems. If $a_1=a_2$, then the replacement leads to  errors typical for the Fourier inversion at points
    of discontinuity. However, in this case, 
     the RHS of \eq{tVq3} can be simplified as follows. We replace $\phi^{\pm,\pm}_q$ with
    $\phi^\pm_q$, which is admissible, then push the line of integration in the inner integral down, cross two simple poles at
    $\xi_1=0$ and $\xi_1=\eta$, and apply the residue theorem.  
    The double integral becomes the following 1D integral:
    \[
    I_2(q;x_1,x_2)=\frac{1}{2\pi}\int_{\Im\eta=\omm}d\eta\, e^{i(x_1-a_2)\eta}\frac{\phipq(\eta)(1-\phimq(\eta))}{-i\eta}.
    \]
    We push the line of integration to $\{\Im\eta=\om_1\}$ and use  the equality 
     $\phipq(\eta)\phimq(\eta)=q/(q+\psi(\eta))$ to obtain the
    formula for the perpetual no-touch option: 
    \beqa\label{no-touch}
   q \tV(f,q;x_1,x_2)&=&    \frac{1}{2\pi}\int_{\Im\xi_1=\om_1}d\xi_1\, \frac{e^{i(x_1-a_2)\xi_1}\phipq(\xi_1)}{-i\xi_1},
    \ x_1\le x_2\le a_2.
    \eqa 
    Of course, \eq{no-touch} can be obtained using the main theorem directly.
        }
    \end{rem}

\begin{rem}\label{rem:two_form_cpdf}{\rm 
   One  can push the line of integration in the outer integral in the double integral on the RHS of \eq{tVq3} up and
  obtain 
  \beqast\label{I2Hilb}
  I_2(q;x_1,x_2)&=&\frac{1}{4\pi}\int_{\Im\xi_1=\om_1}d\xi_1\, e^{i(x_1-a_1)\xi_1}\frac{\phi^{++}_q(\xi_1)\phi^{--}_q(\xi_1)}{-i\xi_1}\\\nonumber
  &&+\frac{1}{(2\pi)^2}\mathrm{v.p.}\int_{\Im\eta=\om_1}d\eta\,e^{i(x_1-a_2)\eta}\phi^{++}_q(\eta)
  \int_{\Im\xi_1=\om_1}d\xi_1\, \frac{e^{i\xi_1(a_2-a_1)}\phi^{--}_q(\xi_1)}{\xi_1(\xi_1-\eta)},
\eqast
where $\mathrm{v.p.}$ denotes the Cauchy principal value. After that, one  can apply the fast Hilbert transform. The integrand decaying very slowly at infinity, accurate calculations are possible only if very long grids are used, hence, the CPU cost is very large even for a moderate error tolerance.    }\end{rem}

  \subsubsection{Example II. Option to exchange the supremum for a power of the underlying} Let $\be>1$. Consider the option to exchange the  supremum $\bar S_T=e^{\bar X_T}$ for the power $S_T^\be=e^{\be X_T}$. The payoff function  $f(x_1,x_2)=(e^{\be x_1}-e^{x_2})_+\bfo_{(-\infty,x_2]}(x_1)$ satisfies \eq{simple_bound_f_X_barX1pr}-\eq{simple_bound_f_X_barX2pr}
 with arbitrary $\muppr>0$, $\mumpr<-\be$. 
 In Sect. \ref{ss:proof_Prop_prop:exch_beta}, we prove

\begin{prop}\label{prop:exch_beta}
 Let $\be>1$ and let 
 conditions of Theorem \ref{thm:mainFT} hold  with $\mum<-\be, \mup> 0$.
 Then, for  $x_1\le x_2$, and any $0<\om_1<\mup$, $\mum<\omm<-\be$,
 \bbe\label{qtVbe2}
 \tV(f;q,x_1,x_2)=I_1(q,x_1,x_2)+q^{-1}\sum_{j=2,3}I_j(q,x_1,x_2),
 \ee
 where
 $I_j(q,x_1,x_2)$, $j=1,2,3,$ are given by 
  \bbe\label{eq:I2}
 I_1(q,x_1,x_2)= \frac{1}{2\pi }\int_{\Im\xi_1=\om_1}d\xi_1\,\frac{e^{i(x_1-x_2)\xi_1}}{q+\psi(\xi_1)}
 \left(\frac{e^{x_2\be}}{\be-i\xi_1}
  +\be \frac{e^{x_2(1+i\xi_1(1-1/\be))}}{(\be-i\xi_1)(-i\xi_1)}-\frac{e^{x_2}}{-i\xi_1}\right),
  \ee
  \bbe\label{eq:I3}
I_2(q,x_1,x_2)=a^-_q\frac{e^{x_2}}{2\pi}\int_{\Im\eta=\omm}d\eta\, e^{i(x_1-x_2)\eta}\frac{\phi^{++}_q(\eta)}{i\eta(1-i\eta)},
\ee 
\beqa\label{eq:I4}
&&
I_3(q,x_1,x_2)\\\nonumber
&=&\frac{1}{(2\pi)^2}\int_{\Im\eta=\omm} d\eta\,e^{i(x_1-x_2)\eta}\phi^{++}_q(\eta)
\int_{\Im\xi_1=\om_1} d\xi_1\, e^{-ix_2\xi_1}\frac{\phi^{--}_q(\xi_1)}{i(\eta-\xi_1)}\\\nonumber
&&\cdot\left[\frac{e^{\be x_2}}{i\eta-\be}
+\frac{\be e^{(1+i\xi_1(1-1/\be)x_2}(1-i\xi_1/\be)}{(\be-i\xi_1)(-i\xi_1)(i\eta-1-i\xi_1(1-1/\be))}
-\frac{e^{x_2}(1-i\xi_1)}{(-i\xi_1)(i\eta-1)}\right].
\eqa
 
\end{prop}

     \section{Numerical evaluation of $V(f;T;x_1,x_2)$ in Example I}\label{s:numer}

     \subsection{Standing assumption}\label{ss:numer_scheme: simple situation}
     In this section, we assume that $X$ is a SINH-regular process of order $\nu\in [0+,2]\setminus\{1+\}$ and type
     $([\mum,\mup], \cC_{\gam,\gap},\cC_{\gampr,\gappr})$, where $\mum<0<\mup$ and $\gampr<0<\gappr$. Furthermore, we assume that either
     $\nu\ge 1$ or $\nu<1$ and the ``drift" $\mu$ in \eq{eq:reprpsi} is 0. Then 
     \bbe\label{cone_nupr}
     \Re \psi(\xi)\ge c_{\psi;\infty} |\xi|^{\nu}-C_\psi,\quad \forall\, \xi\in i(\mum,\mup)+(\cC_{\gampr,\gappr}\cup\{0\}),
     \ee
     where $C_\psi, c_{\psi;\infty}>0$ are independent of $\xi$.
     \begin{lem}\label{lem:cones_Brom} Let the characteristic exponent $\psi$ of a SINH-regular process satisfy \eq{cone_nupr}.
     
     Then there exist $\om_\ell\in (0,\pi/2)$ and $c, \sg>0$ such that for all 
     $q\in \sg+\cC_{\pi/2+\om_\ell}$ and  $\xi\in i(\mum,\mup)+(\cC_{\gampr,\gappr}\cup\{0\})$,
     \bbe\label{bound:two cones}
     |q+\psi(\xi)|\ge c(|q|+|\xi|^\nu).
     \ee
     \end{lem}
     \begin{proof} Since $\psi(\xi)$ admits an upper bound via $C(1+|\xi|^\nu)$, condition \eq{cone_nupr} implies that there
     exist $C_1>0$ and $\ga\in (0,\pi/2)$ such that
     $
     \psi(\xi)+C_1\in \cC_\ga$ for all $\xi\in i(\mum,\mup)+(\cC_{\gampr,\gappr}\cup\{0\})$. Hence, for any $\om_\ell\in (0, \pi/2-\ga)$,
     there exist $\sg,c>0$ such that \eq{bound:two cones} holds.
     
     \end{proof}  
    The bound \eq{bound:two cones}
 allows us to use
     one sinh-deformed contour in the lower half-plane and  one in the upper half-plane for all purposes: the calculation 
     of the Wiener-Hopf factors and evaluation of the integrals on the RHS of \eq{tVq3}. If either $\mum=0$ or $\mup=0$, then both contours
     must cross $i\bR$ in the same half-plane but the types of contours (two non-intersecting contours, one with the wings deformed upwards, the other one with the wings deformed downwards) remain the same as in the case $\mum<0<\mup$.
     
     If \eq{cone_nupr} fails, for instance, if $\nu<1$ and $\mu\neq 0$, then the contour of integration in the formulas for the Wiener-Hopf factors can be deformed only upwards (if $\mu>0$) or downwards (if $\mu<0$). A similar complication
     arises if $\nu=1+$. For instance, for KoBoL of order 1+ in the asymmetric case $c_+\neq c_-$, the type of
     admissible deformations depends on the sign of $c_1-c_2$. Hence, we need to use an additional contour to evaluate the Wiener-Hopf factors.
     Even more importantly,
     the conformal deformations can be used only if the Gaver-Stephest method or GWR algorithm are used or the line of integration in the Bromwich integral is not deformed; conformal deformations  of the
     contours of integration in the formula for $\tV(f;q,x_1,x_2)$ and the Bromwich integral are impossible if we want to preserve the analyticity of the double and triple integrands.  To see this,
     it suffices to consider  the degenerate case $\psi(\xi)=-i\mu\xi$: the conditions $q-i\mu \xi_1\not\in (0,\infty]$,  $q-i\mu\eta\not\in (0,\infty]$ are impossible to satisfy if $\Re q\to -\infty$, and the $\xi_1$- and $\eta$- contours are deformed upward and downward.
 Hence, we can either use the Gaver-Wynn Rho algorithm (see Sect. \ref{GavWynn}) or acceleration schemes of the Euler type,
     e.g., the summation by parts formula (see Sect. \ref{ss:sum-by-parts}).   See Sect. \ref{s:SL_vinvar} for details.
   Finally, if either $\gampr=0$ or $\gappr=0$ (but not both), then additional complications arise, and some of deformations have to be of a less
   efficient sub-polynomial type. See \cite{ConfAccelerationStable} for examples in the context of calculation of stable probability distributions. 
         
  \subsection{Sinh-acceleration}\label{ss:sinh_1D} 
 Consider the first term on the RHS of  \eq{tVq3}, denote it $I_1(q;x_1-a_1)$. As $\xi\to\infty$ along the line of integration,
the integrand decays not faster than $|\xi|^{-3}$. 
The error of the truncation $\sum_{|j|\le N}$ of the infinite sum  $\sum_{j\in \bZ}$  in the infinite trapezoid rule is approximately equal to the error of the truncation 
$\int_{-\La}^\La$, $\La=N\ze$, of the integral $\int_{-\infty}^{+\infty}$, hence, for a small error tolerance $\eps>0$, $\La$ must be of the order of $\eps^{-1/2}$, and the complexity of
the numerical scheme of the evaluation of the integral is  of the order of $\eps^{-1/2}\ln(1/\eps)$. If $x_1-a_1$ is not small
in  absolute value, acceleration schemes of the Euler type can be employed to decrease the number of terms of the simplified trapezoid rule. If $x_1-a_1$ is zero or very close to 0,  Euler acceleration schemes are either non-applicable or rather inefficient.

Let $X$ be SINH-regular.  Assuming that in Definition \ref{def:SINH_reg_proc_1D0}, 
  $\ga_\pm$ are not extremely small in absolute value,  the sinh-acceleration \eq{eq:sinh}
is the most efficient change of variables. Note that 
in \eq{eq:sinh}, 
$\om_1\in \bR$ is, generally, different from $\om_1$ in the formulas of the preceding sections, $\om\in (-\pi/2,\pi/2)$ and $b>0$. The parameters $\om_1,b,\om$ are chosen so that the contour 
$\cL_{\om_1,b,\om}:=\chi_{\om_1,b,\om}(\bR)\subset i(0,\mup)+(\cC_{\gam,\gap}\cup\{0\})$. The parameter $\om$ is chosen so that the oscillating factor becomes a fast decaying one. Under the integral sign of the integral $I_1(q;x_1-a_1)$, the oscillating factor is $e^{i(x_1-a_1)\xi_1}$.
Hence, if $x_1<a_1$, we must choose $\om\in (\gampr, 0)$ (an approximately optimal choice is $\om=\gampr/2$), if $x_1-a_1>0$, we must choose
$\om\in (0,\gappr)$ (an approximately optimal choice is $\om=\gappr/2$), and if $x_1=a_1$, any $\om\in (\gampr,\gappr)$ is admissible 
(an approximately optimal choice is $\om=(\gampr+\gappr)/2$). If $x_1-a_1<0$, it is advantageous
 to push the line of integration in the 1D integral to the lower half-plane, and, on crossing the simple pole at 0,  apply the residue theorem.
 
To evaluate the repeated integral on the RHS of \eq{tVq3}, we deform both lines of integration.
Since $a_2-a_1>0$, it is advantageous to deform the wings of the contour of integration w.r.t. $\xi_1$ up; denote this contour
$\cL^+:=\cL_{\om^+_1,b^+,\om^+}$. Since 
$x_1-a_2\le 0$, it is advantageous to deform the wings of the contour of integration w.r.t. $\eta$ down, denote this contour
$\cL^-:=\cL_{\om^-_1,b^-,\om^-}$. Hence, we choose $\om^+=\gappr/2$, $\om^-=\gampr/2$;
the remaining parameters are chosen so that $\mup>\om^+_1+b^+\sin\om^+>0>\om^-_1+b^-\sin\om^->\mum$. 
See Fig. \ref{fig:TwoGraphsandStrip}.
The result is
 \beqa\label{tVq3sinh}
 \tV(f;q,x_1,x_2)&=&\frac{1}{2\pi}\int_{\cL_{\om_1,b,\om}}\frac{qe^{i(x_1-a_1)\xi_1}}{(q+\psi(\xi_1))(-i\xi_1)}d\xi_1\\\nonumber
&+& \frac{1}{(2\pi)^2 q}\int_{\cL^-}  d\eta\, e^{i(x_1-a_2)\eta}\phipq(\eta)
  \int_{\cL^+}d\xi_1
\, \frac{e^{i\xi_1(a_2-a_1)}\phimq(\xi_1)}{\xi_1(\xi_1-\eta)}.
 \eqa
 We make an appropriate  sinh-change of variables in each integral, and apply the simplified trapezoid rule w.r.t. each new variable.
   
 \subsection{Calculations using the sinh-acceleration in the Bromwich integral}\label{ss:SINH-Bromwich}
 Define
 \bbe\label{eq:sinhLapl}
\chi_{L; \sg_\ell,b_\ell,\om_\ell}(y)=\sg_\ell +i b_\ell\sinh(i\om_\ell+y),
\ee
where $\om_\ell\in (0,\pi/2), b_\ell>0, \sg_\ell-b_\ell\sin\om_\ell>0$, and deform 
 the line of integration in the Bromwich integral to
$\cL^{(L)}=\chi_{L; \sg_\ell,b_\ell,\om_\ell}(R)$. For $q\in \cL^{(L)}$, we can calculate
$\tV(f;q,x_1,x_2)$ using the same algorithm as in the case $q>0$,   
 if there exist $R,q_0,\ga>0$ and $\ga_{--}<0<\ga_{++}$ such that  $q+\psi(\eta)\neq 0$ for all $q\in \cL^{(L)}$ and
  $\eta\in \cC_{\ga_{--},\ga_{++}}, |\eta|\ge R$. In order to avoid the complications of the evaluation of the logarithm on the Riemann surface, it is advisable to ensure  that $1+\psi(\eta)/q\not\in (-\infty,0]$ for pairs $(q,\eta)$ used in the numerical procedure.  See Fig. \ref{Qphipm} for an illustration.
  These conditions can be satisfied if \eq{cone_nupr} holds.
    
  The sequence of deformations is as follows. 
First, for $q$ on the line of integration $\{\Im q=\sg\}$ in the Bromwich integral, we deform the contours of deformation w.r.t. $\eta$ and $\xi_1$ (and contours in the formulas for the Wiener-Hopf factors). 
Then we deform the line of integration w.r.t. $q$ into the contour $\cL^{(L)}$. We choose $\om_\ell$ and $\om^\pm$  sufficiently small in absolute value so that, in the process of deformation,
for all $1+\psi(\xi)/q\neq 0$ and $q+\psi(\eta)/q\neq 0$  for all dual variables $q,\eta,\xi_1$ that appear in the formulas for $\tV(f;q;x_1,x_2)$ and formulas for
the Wiener-Hopf factors.  To make an appropriate choice, the bound \eq{cone_nupr} must be taken into account.
See \cite{paraLaplace} for details. In \cite{paraLaplace}, fractional-parabolic deformations and changes of variables were used.
The modification to the sinh-acceleration is straightforward. 

\subsection{The main blocks of the algorithm}\label{ss:main_blocks} For the sake of brevity, we omit the block for
the evaluation of the 1D integral on the RHS of \eq{tVq3}; this block is  the same  as
in the European option pricing procedure (see \cite{SINHregular}); the type of deformation depends on the sign of $x_1-a_1$. For the 2D integral,
the scheme is independent of $x_1-a_1$. 
We formulate the algorithm assuming that the sinh-acceleration is applied to the Bromwich integral; if the Gaver-Wynn Rho algorithm is used, the modifications of the first step and last step are  described in Sect. \ref{GavWynn}.
We calculate $F(T,a_1,a_2)=V(T,a_1,a_2; 0, 0)$ (that is,  $x_1=x_2=0$).
\begin{enumerate}[Step I.]
\item 
Choose the sinh-deformation in the Bromwich integral and grid for the simplified trapezoid rule: $\vec{y}=\ze_\ell*(0:1:N_\ell)$,
$\vec{q}=\sg_\ell+i*b_\ell*\sinh(i*\om_\ell+\vec{y})$.
Calculate  the derivative $\vec{der_\ell}=i*b_\ell*\cosh(i*\om_\ell+\vec{y})$.
\item
Choose  the sinh-deformations and grids for the simplified trapezoid rule on $\cL^\pm$: $\vec{y^\pm}=\ze^\pm*(-N^\pm:1:N^\pm)$,
$\vec{\xi^\pm}=i*\om_1^\pm+ b^\pm*\sinh(i*\om^\pm+i\vec{y^\pm})$.  Calculate $\vec{\psi^\pm}=\psi(\vec{\xi^\pm})$ and 
$\vec{der^\pm}=b^\pm*\cosh(i*\om^\pm+\vec{y^\pm}).
$
\item
Calculate the matrices $D^+=[1/(\xi^+_j-\xi^-_k)]$ and $D^-=[1/(\xi^-_k-\xi^+_j)]$ (the sizes are $(2*N^++1)\times (2*N^-+1)$
and $(2*N^-+1)\times (2*N^++1)$, respectively).
\item
{\sc The main block (the same block is used if the Gaver-Wynn Rho algorithm is applied).}  For given $x_1,x_2, a_1,a_2$, in the cycle in $q\in \vec{q}$, evaluate
\begin{enumerate}[(1)]
\item
$\phipq$ at points of the grid $\cL^+$ and  $\phimq$ at points of the grid $\cL^-$ 
using \eq{phipq_def}-\eq{phimq_def}: 
\[
\vec{\phi^\pm_q}=\exp\left[((\mp\ze^\pm*i/(2*\pi))*\vec{\xi^\pm}.*(\log(1+ \vec{\psi^\mp}/q)./\vec{\xi^\mp}.*\vec{der^\mp})*D^\pm)\right];
\]
\item
calculate $\phi^\pm_q$ at points of the grid $\cL^\mp$:
$
 \vec{\phi^\pm_{q,\mp}}=q./(q+\vec{\psi^\mp})./\vec{\phi^\mp_q};
 $
 \item
 evaluate the 2D integral on the RHS of \eq{tVq3}
 \beqast
 Int2(q)&=&((\ze^-*\ze^+/(2*\pi)^2)*(\exp(-i*a_2*\vec{\xi^-}).* \vec{\phi^+_{q,-}}.*\vec{der^-})*D^+)\\
 && *\mathrm{conj}((\exp((i*(a_2-a_1))*\vec{\xi^+}).* \vec{\phi^-_{q,+}}/\vec{\xi^+}.*\vec{der^+})').
 \eqast
 \item
 depending on the sign of $x_1-a_1$, use either the arrays $\vec{\xi^+}, \vec{der^+}, \vec{\psi^+}$ or 
 $\vec{\xi^-}, \vec{der^-}, \vec{\psi^-}$ to evaluate $Int1(q)$, the 1D integral on the RHS of \eq{tVq3sinh}.
 \end{enumerate}
 \item
 {\sc Laplace inversion.} Set $Int(\vec{q})=Int2(\vec{q})./\vec{q}+Int1(\vec{q})$, $Int(q_1)=Int(q_1)/2$, and,
 using the symmetry $\overline{\tV(q)}=\tV(\bar q)$, calculate
 \[
V=(\ze_{\ell}/\pi)*\mathrm{real}(\mathrm{sum}(\exp(T*\vec{q}).*Int(\vec{q}).*\vec{der_\ell})). \]

\end{enumerate}

\subsection{Numerical examples}\label{ss:numer}
Numerical results are produced using Matlab R2017b on MacBook Pro, 2.8 GHz Intel Core i7, memory 16GB 2133 MHz. 
The CPU times reported below can be significantly improved because 
\begin{enumerate}[(a)]
\item
the main block of the program, namely, evaluation
of $\tV(q)$ for a given array of $(a_1,a_2)$, is used both for complex and positive $q$'s. However, if $q>0$, we can use the well-known symmetries to decrease the sizes of arrays, hence, the CPU time.
Furthermore, the block admits the trivial parallelization;
\item
we use the same grids for the calculation of the Wiener-Hopf factors $\phi^\pm_q$ and evaluation of integrals on the RHS of  \eq{tVq3sinh}.
However, $\phi^\pm_q$ need to be evaluated only once and used for all points $(a_1,a_2)$. But if $x_1-a_2$ and $a_2-a_1$ are not very small in absolute value, then much shorter grids can be used to evaluate the integrals on the RHS of \eq{tVq3}. See  examples in \cite{iFT,paraHeston,SINHregular,Contrarian}.
Therefore, if the arrays $(x_1-a_2,a_2-a_1)$ are large, then the CPU time can be decreased  using shorter arrays 
for calculation of the integrals on the RHS of  \eq{tVq3sinh}. 
\item
 If the values $F(T,a_1,a_2)$ are needed for several values of $T$
in the range $[T_1,T_2]$, where $T_1$ is not too close to 0 and $T_2$ is not too large, then the CPU time can be significantly decreased
applying  the sinh-acceleration to the Bromwich integral. Indeed, the main step  is time independent,
and the last step, which is the only step where $T$ appears,  admits an easy parallelization. Hence, the CPU time for many values of $T$ is essentially the same as for one value of $T$. 
\end{enumerate}
Item (a), and, partially,  (b) are motivated by our aim to compare the performance of the algorithm based on the Gaver-Wynn
Rho algorithm and the one based on the sinh-acceleration applied to the Bromwich integral.
Since the same subprogram for the evaluation of $\tV(q)$ is used in both cases, and, even in the more complicated second case,
we can achieve the precision of the order of $E-14$, we can safely say that the errors in the first case are the errors of
the  Gaver-Wynn
Rho algorithm itself\footnote{We use the  Gaver-Wynn
Rho algorithm with $M=8$,
hence, 16 positive values of $q$ (depending on $T$) appear. $M=7$ does not work because the error of the  Gaver-Wynn
Rho algorithm itself is too large, $M=9$ does not work because some of the coefficients are so large that
$q\tV(q)$ must be calculated with high precision}; and these errors are of the order of $E-7$  in the cases we considered
(sometimes, larger, in other cases,
somewhat smaller), which agrees with the general empirical observation $(E-0.9)M$, for all choices  
of the parameters of the numerical scheme.  The errors remain essentially the same even if we use much finer and longer grids in the $\eta$- and $\xi$-spaces than it is necessary.
The second motivation for (b) is that we wish to give a relatively short description of the choice of the main parameters of
the numerical scheme.

In the two examples that we consider,
$X$ is KoBoL with the characteristic exponent $\psi(\xi)=c\Gamma(-\nu)(\lp^\nu-(\lp+i\xi)^\nu+(-\lm)^\nu-(-\lm-i\xi)^\nu)$, where
$\lp=1,\lm=-2$ and
(I) $\nu=0.2$, hence, the process is close to Variance Gamma; (II) 
$\nu=1.2$, hence, the process is close to NIG. In both cases, $c>0$ is chosen so that the second instantaneous moment $m_2=\psi^{\prime\prime}(0)=0.1$.
For $X_0=\barX_0=0$, 
we calculate the joint cpdf $F(T,a_1,a_2):=V(T,a_1,a_2; 0,0)$ for $T=0.25$ in Case (I) and for $T=0.05, 0.25,1,5,15$
in Case (II).  In both cases, $a_1$ is in the range $[-0.075,0.1]$ and $a_2$ in the range
$[0.025,0.175]$; the total number of points $(a_1,a_2)$, $a_1\le a_2$, is 44. The parameters of the numerical schemes are chosen as follows.

For SL-processes, and KoBoL is an SL-process,
any sinh-deformation is admissible provided 
is a subset of  $i(\mumpr,\muppr)+(\cC_+\cup\{0\})$ and
$q_1+\psi(i(\om_1-b\sin(\om))>0$ for the smallest $q=q_1>0$ used in the 
 Gaver-Wynn
Rho algorithm. IIf $q_1+\psi(i(\om_1-b\sin(\om))\le0$ then we can reduce the calculations to the case $q_1+\psi(i(\om_1-b\sin(\om))>0$ 
crossing the purely imaginary zero of $q+\psi(\xi)$  as in \cite{paired}.
In the examples that we consider,  $q_1+\psi(i(\om_1-b\sin(\om))>0$. 

For SL-processes, the choice of the most important parameters
$\om^\pm$ trivializes: $\om^\pm=\pm \pi/4\cdot\min\{1,1/\nu\}$, and the half-width $d^\pm$ of the strips of analyticity in the 
new coordinates is $d=|\om^\pm|$. It can be easily shown that, for the Merton model and Meixner processes,
one can choose  $\om^\pm=\pm \pi/8$ and $d^\pm=|\om^\pm|$ (see \cite{EfficientAmenable} for the analysis of the domain of analyticity and zeros of
$q+\psi(\xi)$ for popular L\'evy models). 
Thus, given the error tolerance $\eps$, we can easily write a universal approximate recommendation for the choice of $\ze$.
The recommendation for an approximately optimal choice of the truncation parameter $\La=N\ze$ is the same as in \cite{Contrarian}.
As in \cite{Contrarian}, typically, the recommendation  leads to grids somewhat longer  than necessary. Choosing the parameters by hand, we observe that
the results with the errors of the order of E-7, which are inevitable with the  Gaver-Wynn
Rho algorithm, can be achieved using
the sinh-acceleration in the $\xi$- and $\eta$-spaces, with grids of the length  100 or even smaller (depending on $\nu$ and $T$).
If the calculations are made using the Hilbert transform or simplified trapezoid rule without the conformal deformations, then much longer arrays will be needed (thousand times longer and more) to satisfy even larger error tolerance, and the increase 
of the speed due to the use of the fast Hilbert transform or fast convolution and fast inverse Fourier transform cannot compensate for the very large increase of the sizes of the arrays. 

If the sinh-acceleration in the Bromwich integral is used, then we can satisfy the error tolerance of the order E-14 and smaller  using
the $q$-grids of the order of 100-150, and the $\xi$- and $\eta$-grids of the order of $250-450$. We use two types of deformations: (I) $\om_\ell=(\pi/2)/9, \om_\pm=\pm (\pi/2)/4.5\cdot\min\{1,1/\nu\}$ (``+" for $\cL^+$, ``-" for $\cL^-$) and (II) $\om_\ell=(\pi/2)/10, \om_\pm=\pm (\pi/2)/5\cdot\min\{1,1/\mu\}$. Since each of the three curves has changed, the probability of a random agreement between the two results is negligible.
The differences being less than $E-14$, with some exceptions in the case $T=15$, we take these values  as the benchmark.
The errors in Tables \ref{table:cpdf1} and \ref{table:errors2} are calculated w.r.t. the benchmark probabilities. The CPU time for the benchmark probabilities is in the range 5-8 msec, for one pair $(a_1,a_2)$, and 35-60  msec for 44 points (average of 100 runs). Choosing the parameters by hand, we calculated prices
with errors somewhat smaller than the errors of the  Gaver-Wynn
Rho algorithm. The $\xi$- and $\eta$-grids can be chosen shorter than in
the case of the  Gaver-Wynn
Rho algorithm but the length of the $q$-grid is several times larger than 16 in the  Gaver-Wynn
Rho algorithm;
  the CPU time is several times larger.

 \begin{table}
\caption{\small Joint cpdf $F(T,a_1,a_2):=\bQ[X_T\le a_1, \barX_T\le a_2\ |\ X_0=\barX_0=0]$, and errors (rounded)
and CPU time (in msec) of two numerical schemes. KoBoL 
close to Variance Gamma, with an almost symmetric jump density, and no ``drift": $m_2=0.1$, $\nu=0.2, \lm=-2, \lp=1, \mu=0$;
$T=0.25$.
 }
 {\tiny
\begin{tabular}{c|c|cc|c|cc}
\hline\hline
$a_2/a_1$ & -0.075 & -0.05 & -0.025 & 0 & 0.025 \\\hline
0.025 & 0.0528532412024316 & 0.0649856679446115 & 0.0879014169039594 & 0.506498701211732 & 0.923417160799499\\
0.05 & 0.0533971065051705 & 0.0656207900757611 & 0.088669961239051 & 0.507497961893707 & 0.925278586629321\\
0.075 &  0.0536378889312989 & 0.0658957955144874 & 0.0889908892581364 & 0.50788584329118 & 0.925781540582069\\
0.1 & 0.0537738608706033 & 0.0660488001673674 & 0.0891656084917816 & 0.508089681056682 & 0.926027783268806\\
0.175 & 0.0539603399744032 & 0.0662551510091744 & 0.0893960371866527 & 0.508350135593748 & 0.92632726895684
\\\hline
\end{tabular}

\begin{tabular}{c|ccccc|ccccc}
\hline\hline
&&& $A$ & & && & B & \\\hline
$a_2/a_1$ & -0.075 & -0.05 & -0.025 & 0 & 0.025 & -0.075 & -0.05 & -0.025 & 0 & 0.025
 \\
0.025 &	-1.3E-08	& -1.4E-08 &	-2.0E-08	& 1.6E-05 &	1.5E-08 & 6.9E-11 &	4.6E-11	& 6.1E-11 &	9.5E-08 &	4.5E-09\\
0.05	& -1.4E-08 &	-1.4E-08 &	-1.9E-08 &	3.5E-05 &	1.0E-08 &
4.67E-11 &	1.9E-11 &	2.7E-11 &	9.5E-08 &	2.5E-09
\\
0.075 &	-1.4E-08	&-1.4E-08 &	-1.8E-08 &	-2.7E-05	& 1.0E-08 &
3.8E-11 &	9.2E-12 &	1.6E-11 &	9.5E-08 &	2.5E-09
 \\
0.1 &	-1.4E-08	& -1.3E-08 &	-1.7E-08 &	-6.9E-06 &	1.0E-08 &
3.7E-11 &	8.6E-12 &	1.5E-11 &	9.5E-08 &	2.5E-09
\\
0.175 & -1.3E-08 &	-1.3E-08	& -1.6E-08 &	-7.1E-07	& 1.1E-08
& 3.3E-11 &	3.9E-12 &	9.7E-12 &	9.5E-08 &	2.5E-09\\
\end{tabular}
}
\begin{flushleft}{\tiny
Errors of the benchmark values: better than e-14. CPU time per 1 point: 118, per 44 points: 1,089.\\
A:  Gaver-Wynn
Rho algorithm, $2M=16$, $N^\pm=110$. CPU time per 1 point: 6.4; per 44 points:
44.3.\\
B: SINH applied to the Bromwich integral, with $N_\ell=65, N^\pm=91$. CPU time per 1 point 13.3, per 44 points: 175.\\
If in A, $N^\pm=115$ instead of $N^\pm=110$ are used, the rounded errors do not change but the CPU time increases.\\}
\end{flushleft}

\label{table:cpdf1}
 \end{table}

 \begin{table}
\caption{\small Errors (rounded) and CPU time (in msec) of two numerical schemes for
the calculation of the joint cpdf $F(T,a_1,a_2):=\bQ[X_T\le a_1, \barX_T\le a_2\ |\ X_0=\barX_0=0]$; $T=0.25$. KoBoL 
close to NIG, with an almost symmetric jump density, and no ``drift": $m_2=0.1$, $\nu=1.2, \lm=-2, \lp=1$.
The benchmark values (for $T=0.05, 0.25, 1, 5, 15$) are in Table \ref{table:cpdf2} in Section \ref{ss:figures}. 
 }
 {\tiny
\begin{tabular}{c|ccccc|ccccc}
\hline\hline
&&& $A$ & & && & B & \\\hline
$a_2/a_1$ & -0.075 & -0.05 & -0.025 & 0 & 0.025 & -0.075 & -0.05 & -0.025 & 0 & 0.025 \\
0.025 &	2.6E-07	& 2.3E-06	& -1.1E-06	& -3.1E-06 &	4.0E-06 & 5.3E-09 &	7.5E-09 &	1.1E-08 &	1.6E-08 &	2.8E-08\\
0.05	& 1.5E-06	& 3.9E-06	& -4.9E-06 &	2.6E-07 &	2.4E-06 & 
1.7E-09 &	2.3E-09 &	3.3E-09 &	5.3E-09 &	8.0E-09\\
0.075 &	2.1E-06 &	4.8E-06 &	1.9E-07 &	-1.4E-07 &	4.1E-07 &6.3E-10 &
	8.3E-10 &	1.3E-09 &	8.1E-10 &	9.6E-10
\\
0.1 &	1.7E-06 &	4.6E-06 &	-1.7E-05 &	-1.5E-08 &	3.6E-06 &2.6E-10 &	3.4E-10 &	4.5E-10 &	5.6E-10 &	4.1E-10\\
0.175 & 1.9E-06 &	5.3E-06 &	-6.7E-06 &	6.5E-09 &	3.2E-06 &
3.5E-11 &	4.3E-11 &	5.2E-11 &	2.5E-10 &	1.4E-10\\
\end{tabular}
}
\begin{flushleft}{\tiny
Errors of the benchmark values: better than E-14. CPU time per 1 point: 305, per 44 points: 3,160.\\
A:  Gaver-Wynn
Rho algorithm, $2M=16$, $N^\pm=110$. CPU time per 1 point: 8.7; per 44 points:
58.1.\\
B: SINH applied to the Bromwich integral, with $N_\ell=79, N^\pm=115$. CPU time per 1 point 22.3, per 44 points: 203.\\
If in A, $N^\pm=115$ are used, the rounded errors do not change.\\}
\end{flushleft}

\label{table:errors2}
 \end{table}
 \begin{rem}\label{rem 1/nu}{\rm The factor $\min\{1,1/\nu\}$ is needed to ensure that the image of the strip of analyticity $S_{(-d,d)}$ in the
 $y$-coordinate under the map $y\mapsto q+\psi(\chi_{\om_1,b,\om}(y))$, used to satisfy the error tolerance for the infinite trapezoid rule, does not cross the imaginary axis. 
 }
 \end{rem}
\begin{rem}\label{rem:1nu2}{\rm 
The reader observes that in the case $\nu=0.2$ (process is close to Variance Gamma, Table \ref{table:cpdf1}), the target precision can be achieved at
a smaller computational cost than in the case $\nu=1.2$ (process is close to NIG, Table \ref{table:errors2}). For any method that does not explicitly use the conformal deformation technique, one expects that
the case $\nu=0.2$ must be much more time consuming because the integrands decay much slower than in
the case $\nu=1.2$. However,  we can use a larger 
step in the infinite trapezoid rule  in the case $\nu=0.2$, and the truncation parameter $\La=N\ze$ is essentially the same for all $\nu$ unless $\nu$ is very close to 0.
}
\end{rem}

\section{Conclusion}\label{s:concl}
In the paper,
we derive explicit formulas for the Laplace transforms of expectations of functions of a L\'evy process on $\bR$ and its running supremum,
in terms of the  EPV operators $\cE^\pm_q$ (factors in the operator form of the Wiener-Hopf factorization). If the explicit formulas can be efficiently realized for $q$'s used in a  numerical realization of the Bromwich integral,
then the expectations can be efficiently calculated. Standard applications to finance are options with barrier and lookback features, 
with flat barriers.
In the paper, we consider in detail numerical realizations for wide classes of L\'evy processes with the characteristic exponents admitting analytic continuation to a strip around or adjacent to the real axis,
equivalently, with the L\'evy density of either positive or negative jumps decaying exponentially at infinity. Thus, we allow for a stable L\'evy component of negative jumps\footnote{A polynomially decaying stable L\'evy tail is important for applications to risk management, however, from
the computational point of view, the cases of two exponentially decaying tails and only one exponentially decaying tail are essentially indistinguishable.}. The numerical part of the paper is a two-step procedure. First, we derive explicit formulas in terms of a sum of 
1D-3D integrals;
in many cases of interest, the triple integrals are reducible to double integrals over the Cartesian product of two flat contours in the complex plane. As applications,  we calculate the cpdf of the L\'evy process and its supremum $\barX$
and the price of the option to exchange $e^{\barX_T}$ for a power $e^{\be X_T}$.

The repeated integrals can be calculated using the simplified trapezoid rule and the Fast Fourier transform technique (or fast convolution or fast Hilbert transform) if
the expectations need to be calculated at many points in the state space. 
In popular L\'evy models, the characteristic exponent admits analytic continuation to a union of a strip and cone around or adjacent to
the real line. Then the computational cost can be decreased manifold using the conformal deformation technique. We  use the most efficient version: the sinh-acceleration, and explain how the deformations of several contours should be made: two contours for each $q>0$ used in the  Gaver-Wynn
Rho algorithm, and three contours if
the sinh-acceleration method is applied to the Bromwich integral. 
Numerical examples demonstrate the efficiency of the method; the conformal deformation technique applied to the Bromwich integral
achieves the precision of the order of E-14 and the Gaver-Wynn Rho algorithm - of the order of E-08-E-06. However, the latter is faster. Note that  Talbot's deformation cannot be applied if the conformal deformations technique is applied to the integrals with respect to the other dual variables.

In the accompanying papers \cite{EfficientDiscExtremum,EfficientZsufficient,EfficientStableLevyExtremum,EfficientDoubleBarrier},
 the method, results and proofs of the paper are
modified for random walks, barrier and lookback options with discrete monitoring in particular,
pricing barrier and lookback options in stable L\'evy models and double barrier options. 

The methodology of the paper can be extended in several directions, and adapted to
\begin{enumerate}
\item
Monte-Carlo simulations of the joint distribution of a L\'evy process and its extremum, similarly to
\cite{SINHregular,ConfAccelerationStable}, where an efficient procedure for the simulation of the distributions
of L\'evy processes is constructed;
 
\item
American options and  barrier and lookback options with time-dependent barriers, similarly to 
\cite{amer-put-levy,IDUU,amer-reg-sw-SIAM};
\item
 regime-switching L\'evy models, with different payoff functions in different states, similarly to 
\cite{ExitRSw,stoch-int-rate-CF,MSdouble};
\item
models with stochastic volatility and stochastic interest rates. The first step, namely, approximation by regime-switching models, is
the same as in \cite{stoch-int-rate-CF,SVolSSRN,BLHestonStIR08};
\item
models with stochastic interest rates, when the eigenfunction expansion is used to approximate the action of the infinitesimal generator
of the process for the interest rates \cite{BarrStIR};
\item
models with non-standard payoffs arising in applications to real options and Game Theory \cite{no-remorse,monopoly,preemptGames};
\item
multi-factor L\'evy models.
\end{enumerate}
In the case of pricing barrier options, the main blocks of the induction procedures can be replaced with the main block of this paper
(adjusted to the case of more general payoffs); in the case of American options, the iteration procedure at each time step cannot be applied because when the calculations are in the dual space, the positivity of the approximation to the transition operator 
is impossible to guarantee, and the iteration procedure for an approximation to the early exercise boundary at each time step
is justified only if the approximation to the transition operator is positive. Hence, the main block in the present paper can be applied
only if the time step is chosen sufficiently small and no iteration procedure at each time step is used.

  \bibliography{thebibliographyFM21}{}
\bibliographystyle{plain}

\appendix

\section{Technicalities}\label{s:tech}

\subsection{Decomposition of the Wiener-Hopf factors}\label{WHF_decomp}
The following more detailed properties of the Wiener-Hopf factors are established in \cite{NG-MBS,BLSIAM02,barrier-RLPE}
for the class of RLPE (Regular L\'evy processes of exponential type); the proof for SINH-regular processes is the same only $\xi$ is allowed to tend to $\infty$ not only in the strip of analyticity but in the union of a strip and cone. See \cite{BIL,asymp-sens,paired} for the proof of the statements below for
several classes of SINH-regular processes (the definition of the SINH-regular processes formalizing properties used in \cite{BIL,asymp-sens,paired} was suggested in \cite{SINHregular} later.). The contours  in Lemma \ref{lem:atoms} below
are in a domain of analyticity s.t. $q-i\mu\xi\neq 0$ and $1+\psi^0(\xi)/(q-i\mu\xi)\not\in (-\infty,0]$. The latter condition is needed when $\psi^0(\xi)=O(|\xi|^\nu)$ as $\xi\to\infty$ in the domain of analyticity and $\nu<1$. Clearly, in this case, for sufficiently large $q>0$, the condition holds.
In the case of RLPE's, the contours of integration in the lemma below are straight lines in the strip of analyticity 

\begin{lem}\label{lem:atoms}
Let $\mum<0<\mup$,  $q>0$, let $X$ be SINH-regular  of  type 
$((\mum,\mup), \cC_{\gam,\gap}, \cC_{\gampr,\gappr})$, $\mum<0<\mup$, and order $\nu$.  Then
\begin{enumerate}[(1)]
\item
if  $\nu\in [1,2]$ or $\nu\in (0,1)$ and the ``drift" is $\mu=0$, then neither
 $\barX_{T_q}$ nor $\uX_{T_q}$ has an atom at 0, and $\phi^\pm_q(\xi)$ admit the bounds \eq{WHFdecayP} and 
 \eq{WHFdecayM}, 
 where $\nu_\pm>0$ and $C_\pm(q)>0$ are independent of $\xi$;
 \item
  if  $\nu\in (0,1)\cup\{0+\}$  and $\mu>0$, then
 \begin{enumerate}[(a)]
 \item
 $\barX_{T_q}$ has no atom at 0 and $\uX_{T_q}$ has an atom  $a^-_q\de_0$ at zero, where
 \bbe\label{eq:cmqp}
a^-_q=\exp\left[\frac{1}{2\pi}\int_{\cL^+_{\om_1,b,\om}}\frac{\ln((1+\psi^0(\eta)/(q-i\mu\eta))}{\eta}d\eta\right],
\ee
and $\cL^+_{\om_1,b,\om}$ is a contour as in Lemma \ref{lem:WHF-SINH} (ii), lying above $0$;
 \item
for $\xi$ and $\cL^-_{\om_1,b,\om}$ in Lemma \ref{lem:WHF-SINH} (i), $\phipq(\xi)$ admits the representation
 \bbe\label{phip1finvar}
 \phipq(\xi)=\frac{q}{q-i\mu\xi}\exp\left[\frac{1}{2\pi i}\int_{\cL^-_{\om_1,b,\om}}\frac{\xi\ln(1+\psi(\eta)/(q-i\mu\eta))}{\eta(\xi-\eta)}d\eta\right],
 \ee
and  satisfies the bound \eq{WHFdecayP} with $\nup=1$;
 \item
 $\phimq(\xi)=a^-_q+\phi^{--}_q(\xi)$, where $\phi^{--}_q(\xi)$ satisfies 
 \eq{WHFdecayM} with arbitrary $\num\in (0,1-\nu)$;
 \item
 $\cEmq=a^-_qI+\cE^{--}_q$, where $\cE^{--}_q$ is the PDO with the symbol $\phi^{--}_q(\xi)$;
 \end{enumerate}
  \item
  if  $\nu\in (0,1)\cup\{0+\}$  and $\mu<0$, then
 \begin{enumerate}[(a)]
 \item
 $\uX_{T_q}$ has no atom at 0 and $\barX_{T_q}$ has an atom $a^+_q\de_0$ at zero, where
 \bbe\label{eq:cpqp}
a^+_q=\exp\left[\frac{1}{2\pi}\int_{\cL^-_{\om_1,b,\om}}\frac{\ln((1+\psi^0(\eta)/(q-i\mu\eta))}{\eta}d\eta\right],
\ee
and $\cL^-_{\om_1,b,\om}$ is a contour as in Lemma \ref{lem:WHF-SINH} (i), lying below $0$;
 \item
for $\xi$ and $\cL^+_{\om_1,b,\om}$ in Lemma \ref{lem:WHF-SINH} (ii), $\phimq(\xi)$ admits the representation
 \bbe\label{phim1finvar}
 \phimq(\xi)=\frac{q}{q-i\mu\xi}\exp\left[-\frac{1}{2\pi i}\int_{\cL^+_{\om_1,b,\om}}\frac{\xi\ln(1+\psi(\eta)/(q-i\mu\eta))}{\eta(\xi-\eta)}d\eta\right],
 \ee
and
 satisfies the bound \eq{WHFdecayM} with $\num=1$;
 \item
 $\phipq(\xi)=a^+_q+\phi^{++}_q(\xi)$, where $\phi^{++}_q(\xi)$ satisfies  
 \eq{WHFdecayP} with arbitrary $\nup\in (0,1-\nu)$;
 \item
 $\cEpq=a^+_qI+\cE^{++}_q$, where $\cE^{++}_q$ is the PDO with the symbol $\phi^{++}_q(\xi)$;
 \end{enumerate}
 \end{enumerate}
\end{lem}

  \subsection{Proof of bounds \eq{bound_L1_1} and \eq{getaxi}}\label{ss:proof_lemm_g}
   First, we prove
   that  if $a,b>0$, then
  $g_{a,b}$ defined by  $g_{a,b}(\xi_1,\xi_2)=(1+|\xi_1+\xi_2|)^{-a}(1+|\xi_1|)^{-1-b}(1+|\xi_2|)^{-1}$ is of class $L_1(\bR^2)$.
   Consider separately regions $U_j\subset \bR^2$, $j=1,2,3,$
  defined by inequalities $|\xi_2|\le |\xi_1|/2$; $|\xi_2|\ge 2|\xi_1|$; $|\xi_1|/2\le|\xi_2|\le 2|\xi_1|$, respectively. On $U_1$,
  \[
  g_{a,b}(\xi_1,\xi_2)\le C_1(1+|\xi_1|)^{-1-a-b}(1+|\xi_2|)^{-1}\le C_2(1+|\xi_1|)^{-1-b}(1+|\xi_2|)^{-1-a},
  \]
  and the function on the RHS is of class $L_1(\bR^2)$. On $U_2$, $g_{a,b}(\xi_1,\xi_2)$ admits an upper bound via
  the same function (and a different constant $C_2$). Finally,
  \[
  \int_{U_3}d\xi_1d\xi_2\, g_{a,b}(\xi_1,\xi_2)\le C_3\int_{\bR} d\xi_1\, \ln(2+|\xi_1|)(1+|\xi_1|)^{-1-b}<\infty,
  \]
  which proves \eq{bound_L1_1}. To prove \eq{getaxi},
   we consider the restrictions of $g$ to the 
   regions $U_j\subset \bR^3$, $j=1,2,3,$
    defined by the inequalities $|\eta|\le |\xi_1+\xi_2|/2$; $|\eta|\ge 2|\xi_1+\xi_2|$; $|\xi_1+\xi_2|/2\le|\eta|\le 2|\xi_1+\xi_2|$.
    On $U_1$,
    \beqast
     |g(\eta,\xi_1,\xi_2)|&\le& C_1(1+|\eta|)^{-\nup}(1+|\xi_1+\xi_2|)^{-1}(1+|\xi_1|)^{-1-\num}(1+|\xi_2|)^{-1}\\
     &\le& C_2(1+|\eta|)^{-\nup/2-1}(1+|\xi_1+\xi_2|)^{-\nup/2}(1+|\xi_1|)^{-1-\num}(1+|\xi_2|)^{-1},
     \eqast
    on $U_2$,
    \beqast
     |g(\eta,\xi_1,\xi_2)|&\le& C_1(1+|\eta|)^{-\nup-1}(1+|\xi_1|)^{-1-\num}(1+|\xi_2|)^{-1}\\
     &\le& C_2(1+|\eta|)^{-\nup/2-1}(1+|\xi_1+\xi_2|)^{-\nup/2}(1+|\xi_1|)^{-1-\num}(1+|\xi_2|)^{-1}.
     \eqast
  In each case,  the function on the RHS' is of the form $C(1+|\eta|)^{-1-\nup/2}g_{\nup/2,\num}(\xi_1,\xi_2)$, hence, of class $L_1(\bR^3)$. To prove the integrability of $g$ on $U_3$, it suffices to note that
    \beqast
    \int_{|\xi_1+\xi_2|/2\le |\eta|\le 2|\xi_1+\xi_2|}d\eta\, |g(\eta,\xi_1,\xi_2)|&\le& C_3\ln(2+|\xi_1+\xi_2|)g_{\nup,\num}(\xi_1,\xi_2),
    \eqast
    and the RHS admits an upper bound via $C_4 g_{\nup/2,\num/2}(\xi_1,\xi_2)$.
    
    \subsection{Proof of Proposition \ref{prop:exch_beta}}\label{ss:proof_Prop_prop:exch_beta}
We apply Theorem \ref{thm:mainFT} with
 $\muppr\in (0,\mup)$, $\mumpr\in (\mum,-\be)$.
For $x_2> 0$ and $\xi\in\bC$, 
 \beqast\label{widetildef1be}
  (\widehat{f_+})_1(\xi_1, x_2)&=&\int_{x_2/\be}^{x_2}e^{-ix_1\xi_1}(e^{\be x_1}-e^{x_2})dx_1\\
  &=&e^{-ix_2\xi_1}\left(\frac{e^{x_2\be}}{\be-i\xi_1}
  +\be \frac{e^{x_2(1+i\xi_1(1-1/\be))}}{(\be-i\xi_1)(-i\xi_1)}-\frac{e^{x_2}}{-i\xi_1}\right),
  \eqast
  hence, the first term on the RHS of \eq{eq:tVqmain} equals the integral on the RHS of \eq{eq:I2}.
Then we calculate 
\beqast
w_0(y,x_2)&=&\bfo_{[x_2,+\infty)}(y)((e^{\be y}-e^{y})-(e^{\be y}-e^{x_2}))=\bfo_{[x_2,+\infty)}(y)(e^{x_2}-e^{y}),
\\
\widehat{w_0}(\eta,x_2)&=&\int_{x_2}^{+\infty} e^{-iy\eta}(e^{x_2}-e^{y})dy=
 \frac{e^{x_2-ix_2\eta}}{i\eta(1-i\eta)},
\eqast
and obtain that the second term on the RHS of \eq{eq:tVqmain} equals the RHS of \eq{eq:I3}.
Next, we calculate $\hat w^-(q,\eta,x_2)$:
\beqast
\hat w^-(q,\eta,x_2)&=&\int_{x_2}^{+\infty}e^{-iy\eta}\frac{1}{2\pi}\int_{\Im\xi_1=\om_1}d\xi_1\,e^{iy\xi_1}\phi^{--}_q(\xi_1)
\left[\frac{e^{(\be-i\xi_1)y}-e^{(\be-i\xi_1)x_2}}{\be-i\xi_1}\right.\\
&&\left.+\be\frac{e^{(1-i\xi_1/\be)y}-e^{(1-i\xi_1/\be)x_2}}{(\be-i\xi_1)(-i\xi_1)}-\frac{e^{(1-i\xi_1)y}-e^{(1-i\xi_1)x_2}}{-i\xi_1}\right]\\
&=&\frac{e^{-ix_2\eta}}{2\pi}\int_{\Im\xi_1=\om_1}\phi^{--}_q(\xi_1)\left[\frac{e^{(\be-i\xi_1)x_2}}{\be-i\xi_1}\left(\frac{1}{i(\eta-\xi_1)-(\be-i\xi_1)}-\frac{1}{i(\eta-\xi_1)}\right)\right.\\
&&\hskip2.5cm +\frac{\be e^{(1-i\xi_1/\be)x_2}}{(\be-i\xi_1)(-i\xi_1)}\left(\frac{1}{i(\eta-\xi_1)-(1-i\xi_1/\be)}-\frac{1}{i(\eta-\xi_1)}\right)
\\
&&\hskip2.5cm\left.-\frac{e^{(1-i\xi_1)x_2}}{-i\xi_1}\left(\frac{1}{i(\eta-\xi_1)-(1-i\xi_1)}-\frac{1}{i(\eta-\xi_1)}\right)\right]\\
&=&\frac{e^{-ix_2\eta}}{2\pi}\int_{\Im\xi_1=\om_1}d\xi_1\,\frac{\phi^{--}_q(\xi_1)}{i(\eta-\xi_1)}\left[\frac{e^{(\be-i\xi_1)x_2}}{i\eta-\be}
\right.\\
&&\hskip1.5cm \left.+\frac{\be e^{(1-i\xi_1/\be)x_2}(1-i\xi_1/\be)}{(\be-i\xi_1)(-i\xi_1)(i\eta-1-i\xi_1(1-1/\be))}-\frac{e^{(1-i\xi_1)x_2}(1-i\xi_1)}{(-i\xi_1)(i\eta-1)}\right],
\eqast
and, finally, derive the representation \eq{eq:I4} for the double integral on the RHS of \eq{eq:tVqmain}.

     \subsection{General remarks on numerical Laplace inversion}\label{ss:gen_rem}
     The final result is obtained applying a chosen numerical Laplace inversion procedure to
     $\tV(q,\cdot,\cdot)$ defined by  \eq{tVq3}.
     The methods that we construct (main texts: \cite{paraLaplace,paired,BarrStIR,Contrarian}) can be regarded as  further steps in  a general program of study of the efficiency
of combinations of one-dimensional inverse transforms for high-dimensional inversions systematically pursued by
Abate-Whitt, Abate-Valko \cite{AbWh,AbWh92OR,AbateValko04,AbValko04b,AbWh06} and other authors.
Additional methods can be found in \cite{stenger-book}.
Abate and Valko and Abate and Whitt  consider three main
different one-dimensional algorithms for the numerical realization of the Bromwich
integral:
(1) Fourier series expansions with Euler summation (the summation-by-part formula
in Sect. \ref{ss:sum-by-parts}  can be regarded as a special case of Euler summation);
 (2)
combinations
of Gaver functionals, and
(3) deformation of the contour in the Bromwich
integral. Talbot's contour deformation
$q=r\theta(\cot\theta+i), -\pi<\theta<\pi$, is suggested,
and various methods of multi-dimensional inversion
based on combinations of these three basic blocks are discussed. 
It is stated that
for the popular Gaver-Stehfest method, 
the required system precision is about $2.2*M$, and
about $0.9*M$ significant digits are produced for $f(t)$ with good transforms. ``Good" means that $f$ is of class $C^\infty$, and
the transform's singularities are on the negative real axis. If the transforms are not good, then the number of significant digits may be not
 so great and may be not  proportional to $M$. In our previous publications
\cite{paraLaplace,paired}, we develop numerical methods for pricing barrier and lookback options based on the fractional-parabolic deformations, and observed that when we were able to evaluate $\tV(q)$ with the precision E-10 and better, the Gaver-Stehfest method
with $M=8$,
produced fairly accurate results (errors of the order of E-4 or even E-5) although, according to
the general remark in  \cite{AbWh06}, $\tV(q)$'s had to be calculated with the precision E-15. However, in many cases,
the fractional-parabolic acceleration requires too long grids and the accumulation of errors 
of the calculation of the Wiener-Hopf factors leads to the failure of the Gaver-Stehfest method. If the simplified trapezoid rule, without acceleration,
is applied to   the integral under the exponential sign on the RHS' of \eq{phip1} and \eq{phim1}, then the arrays of the size of the order of $10^{9}$ and more are needed. Hence, sufficiently accurate calculations  (nothing to say fast)  are impossible. Indeed, the integrands decay slower than $|\eta|^{-2}$ as $\eta\to \infty$ in the strip of analyticity.

In \cite{MSdouble,single,BLdouble}, it is demonstrated that Carr's randomization (equivalently, the method of lines) allows one to calculate prices of single and double barrier options and barrier options in regime-switching models
with the precision of the order of E-02-E-03 because Carr's randomization procedure works even if 
the calculations at each time step are with the precision of the order of E-04-E-05 only. In \cite{single,BLdouble}, calculations are relatively fast
because grids of different sizes
for the evaluation of the Wiener-Hopf factors and fast convolution at each time step and  the refined version of the inverse FFT (iFFT)
constructed in \cite{single}  are used (standard iFFT and fractional iFFT do not suffice in the majority of cases). 
In \cite{MSdouble}, regime-switching hyper-exponential jump diffusion models are considered, hence, the Wiener-Hopf factors
are easy to calculate with the precision E-14.

In the present paper, as in \cite{Contrarian}, we use the sinh-acceleration to evaluate the Wiener-Hopf factors. 
The summation of several hundreds of terms suffices to achieve the precision better than E-15, hence, the effect of accumulation of machine errors is insignificant, and we can calculate $\tV(q)$ with the precision E-14 and better. Thus, the errors of our method that uses 
 the  Gaver-Wynn
Rho algorithm which we document are the errors of the GWR algorithm itself. These errors are in the range
E-05-E-8, depending on the parameters of the model, $T$ and $a_1,a_2$. For the sake of brevity, we produce the results for
$x_1=x_2=0$; $T, a_1$ and $a_2$ vary.

More accurate results are obtained when we apply the sinh-acceleration to the Bromwich integral. The CPU cost increases several times because the number of $q$'s used is several times larger; but we can achieve the precision E-14 and better. 
Note that Talbot's deformation \cite{Talbot79} is not applicable together with the sinh-deformations
of the other contours of integration, hence, the CPU time is significantly larger and good precision is impossible to achieve in many cases when the sinh-acceleration is very efficient. 
Hence, the best two versions are: the Gaver-Wynn Rho algorithm, if the accuracy of the final result of the order 
of E-6 is admissible, and the sinh-acceleration applied to the Bromwich integral if a higher precision is needed. In both cases,
the Wiener-Hopf factors and $\tV(q)$'s are calculated using the sinh-acceleration.

The sinh-acceleration is similar to but simpler to apply than the saddle-point method (see., e.g., \cite{fedoryuk});  the rate of convergence is approximately the same.  The former method is
more flexible than the latter, in applications to repeated integrals especially. The rate of convergence is approximately the same, and the calculation of individual terms in
numerical realizations is much simpler and less time consuming. Talbot's deformation \cite{Talbot79} is not applicable together with the sinh-deformations
of the other contours of integration, hence, the CPU time is significantly larger and good precision is impossible to achieve in many cases when the method of the paper is vey efficient.
 
   \subsection{Infinite trapezoid rule}\label{infTrap} 

Let $g$ be analytic in the strip
$S_{(-d,d)}:=\{\xi\ | \Im\xi\in (-d,d)\}$ and decay at infinity sufficiently fast so that
$\lim_{A\to \pm\infty}\int_{-d}^d |g(i a+A)|da=0,$
and 
\bbe\label{Hnorm}
H(g,d):=\|g\|_{H^1(S_{(-d,d)})}:=\lim_{a\downarrow -d}\int_\bR|g(i a+ y)|dy+\lim_{a\uparrow d}\int_\bR|g(i a+y)|dy<\infty
\ee
is finite. We write $g\in H^1(S_{(-d,d)})$. The integral
$I=\int_\bR g(\xi)d\xi$
can be evaluated using the infinite trapezoid rule
\bbe\label{inftrap}
I\approx \ze\sum_{j\in \bZ} g(j\ze),
\ee
where $\ze>0$. 
The following key lemma is proved in \cite{stenger-book} using the heavy machinery of sinc-functions. A simple proof can be found in
\cite{paraHeston}.
\begin{lem}[\cite{stenger-book}, Thm.3.2.1]
The error of the infinite trapezoid rule admits an upper bound 
\bbe\label{Err_inf_trap}
{\rm Err}_{\rm disc}\le H(g,d)\frac{\exp[-2\pi d/\ze]}{1-\exp[-2\pi d/\ze]}.
\ee
\end{lem}
Once
an approximately bound for $H(g,d)$ is derived, it becomes possible to satisfy the desired error tolerance
with a good accuracy.  

\subsection{Summation by parts}\label{ss:summation_by_parts}\label{ss:sum-by-parts}
 The rate of decay of the series can be significantly increased
if the infinite trapezoid rule 
 is of the form 
\[
I(a)=\ze\sum_{j\in \bZ} e^{-i aj\ze}g(j\ze),
\]
where $a\in\bR\setminus 0$, and $g'(y)$ decreases faster than $g(y)$ as $y\to\pm\infty$. 
Indeed, then, by the mean value theorem,  
the finite differences $\De g_j=(\De g)(j\ze)$, where $(\De g)(\xi)=g(\xi+\ze)-g(\xi)$, decay faster 
than $g(j\ze)$ as $j\to\pm\infty$ as well. 

The summation by parts formula is as follows. Let $e^{i a\ze}-1\neq 0$. Then
\[
\ze\sum_{j\in \bZ} e^{-i aj\ze }g(j\ze)=\frac{\ze}{e^{i a\ze}-1}\sum_{j\in \bZ} e^{-i aj\ze }\De g_j.
\]
If additional differentiations further increase the rate of decay of the series as $j\to\pm\infty$, then the summation by part procedure can be iterated:
\bbe\label{e:sum_by_part}
\ze\sum_{j\in \bZ} e^{-i aj\ze }g_j=\frac{\ze}{(e^{i a \ze}-1)^n}\sum_{j\in \bZ} e^{- i aj\ze }\De^n g_j.
\ee
After the summation by parts, the series on the RHS of \eq{e:sum_by_part} needs to be truncated.
The truncation parameter can be chosen using the following lemma.
\begin{lem}\label{lem:err_sum_by_parts}
Let $n\ge 1, N>1$ be  integers, $\ze>0, a\in \bR$ and $e^{i a\ze }-1\neq 0$. 

Let $g^{(n)}$ be continuous and let the function
$\xi\mapsto G_n(\xi,\ze):=\max_{\eta\in [\xi,\xi+n\ze]}|g^{(n)}(\eta)|$ be in $L_1(\bR)$.
Then
\beqa\label{trunc_err_sum_by_parts_pos}
\left|\frac{\ze}{(e^{i a\ze}-1)^n} \sum_{j\ge N} e^{-i aj\ze }\De^n g_j\right|&\le& \left(\frac{\ze}{|e^{i a\ze}-1|}\right)^n\int_{N\ze}^{+\infty}G_n(\xi,\ze)d\xi,
\\\label{trunc_err_sum_by_parts_neg}
\left|\frac{\ze}{(e^{i a\ze}-1)^n} \sum_{j\le -N} e^{-i a j\ze }\De^n g_j\right|&\le& \left(\frac{\ze}{|e^{i a\ze}-1|}\right)^n\int_{-\infty}^{-N\ze}G_n(\xi,\ze)d\xi.
\eqa
\end{lem}
\begin{proof} Using the mean value theorem, we obtain
\[
|(\De^n g)(\xi)|\le \ze \max_{\xi_1\in [\xi,\xi+\ze]}|(\De^{n-1} g')(\xi_1)|\le\cdots \le\ze^n\max_{\eta\in [\xi,\xi+n\ze]}|g^{(n)}(\eta)|.
\]
\end{proof}

  \subsection{Gaver-Wynn Rho algorithm}\label{GavWynn} 
  The inverse Laplace transform $V(T)$  of $\tilde V$  is approximated by 
\begin{equation}\label{GS31}
V(T, M)=\frac{\ln(2)}{t}\sum_{k=1}^{2M}\zeta_k \tV\left(\frac{k\ln(2)}{T}\right),
\end{equation}
where $M\in \bN$,
\begin{equation}\label{GS32}
\zeta_k(t, M)=(-1)^{M+k}\sum_{j=\lfloor (k+1)/2\rfloor}^{\min\{k,M\}}\frac{j^{M+1}}{M!}\left(\begin{array}{c} M \\ j\end{array}\right)
\left(\begin{array}{c} 2j \\ j\end{array}\right)\left(\begin{array}{c} j \\ k-j\end{array}\right)
\end{equation}
and $\lfloor a \rfloor$ denotes the largest integer that is less than or equal to  $a$.
If $T$ is large which in applications to option pricing means options of long maturities, then $q=k\ln(2)/T$ is small. In the present paper, efficient calculations of $\tV(f;q,x_1,x_2)$ are possible if $q\ge \sg$, where $\sg>0$ is determined by the parameters of the process and payoff function. Hence, if $T$ is large, we modify \eq{tVBrom}
\bbe\label{tVBrom_a}
V(f;T;x_1,x_2)=\frac{e^{aT}}{2\pi i}\int_{\Re q=\sg}e^{qT}\tV(f;q+a;x_1,x_2)\,dq,
\ee
where $a>0$ is chosen 
so that $\ln(2)/T+a>\max\{-\psi(i\mumpr),-\psi(i\muppr)\}$.
 In the paper, as in \cite{paired,Contrarian}, we apply Gaver-Wynn-Rho (GWR) algorithm, which is
  more stable than the Gaver-Stehfest method.
  
  Given a converging sequence $\{f_1, f_2,
\ldots\}$, Wynn's algorithm estimates the limit $f=\lim_{n\to\infty}f_n$ via $\rho^1_{N-1}$, where $N$ is even,
and $\rho^j_k$, $k=-1,0,1,\ldots, N$, $j=1,2,\ldots, N-k+1$, are calculated recursively as follows:
\begin{enumerate}[(i)]
\item
$\rho^j_{-1}=0,\ 1\le j\le N;$
\item
$\rho^j_{0}=f_j,\ 1\le j\le N;$
\item
in the double cycle w.r.t. $k=1,2,\ldots,N$, $j=1,2,\ldots, N-k+1$, calculate
\[
\rho^j_{k}=\rho^{j+1}_{k-2}+k/(\rho^{j+1}_{k-1}-\rho^{j}_{k-1}).
\]
We apply Wynn's algorithm with the Gaver functionals
\[
f_j(T)=\frac{j\ln 2}{T}\left(\frac{2j}{j}\right)\sum_{\ell=0}^j (-1)^j\left(\frac{j}{\ell}\right)\tilde f((j+\ell)\ln 2/T).
\]
 \end{enumerate}

   \subsection{Calculations in the case of finite variation processes with non-zero drift}\label{s:SL_vinvar}
 \subsubsection{Gaver-Wynn Rho algorithm is used}\label{ss:GavWynn_pos_mu}
 Consider 1D integral on the RHS of \eq{tVq3}. 
 
 (I-) If $x_1-a_1\le 0$, it is advantageous to deform the line of integration downwards. Hence, the contour $\cL_{\om_{10}, b_0, \om_0}$ in the new integral is defined by $\om_0<0$,
 and $\om_{10}\in \bR, b_0>0$ such that that $\sg_0:=\Im \psi(\chi_{\om{10},b_0,\om_0}(0))=\om_{10}+b_0\sin\om_0\in (0,\mup)$ and
 $q+\psi(\chi_{\om{10},b_0,\om_0}(0)=q+\mu\sg_0+\psi^0(i\sg_0))>0$. Alternatively, one can push the line of integration
 below 0, apply the residue theorem (the additional term $1/q$ appears), and choose $\om_0<0$,
 and $\om_{10}\in \bR, b_0>0$ so that that $\sg_0:=\Im \psi(\chi_{\om{10},b_0,\om_0}(0))=\om_{10}+b_0\sin\om_0\in (\mum,0)$ and
 $q+\psi(\chi_{\om{10},b_0,\om_0}(0))=q+\mu\sg_0+\psi^0(i\sg_0))>0$. 
  
  (I+) If $x_1-a_1> 0$, it is advantageous to deform the line of integration upwards. Hence, the contour $\cL_{\om_{10}, b_0, \om_0}$ in the new integral is defined by $\om_0>0$,
 and $\om_{10}\in \bR, b_0>0$ such that that $\sg_0:=\Im \psi(\chi_{\om{10},b_0,\om_0}(0))=\om_{10}+b_0\sin\om_0\in (0,\mup)$ and
 $q+\psi(\chi_{\om{10},b_0,\om_0}(0))=q+\mu\sg_0+\psi^0(i\sg_0)>0$. 
  
(II) Now we consider the 2D integral. Since $x_1-a_2<0$, it is advantageous to deform the outer line of integration downwards. Hence, the contour $\cL_{\om_{1-}, b_-, \omm}$ in the new integral is defined by $\omm<0$,
 and $\om_{1-}\in \bR, b_->0$ such that that $\sg_-:=\Im \psi(\chi_{\om_{1-},b_-,\omm}(0))=\om_{1-}+b_-\sin\omm\in (\mum,0)$ and
 $q+\psi(\chi_{\om{1-},b_-,\omm}(0))=q+\mu\sg_-+\psi^0(i\sg_-)>0$. Both conditions can be satisfied choosing suffciently small
 (in absolut value)  $\om_{1-}$ and $b_-$. The inner contour is deformed upward, and the same contour as in the case (I+) can be used.

 \subsubsection{Infinite trapezoid rule applied to the Bromwich integral}\label{ss:inf_trap_Bromwich}
 After the infinite trapezoid rule is applied, one can use the summation-by-parts procedure (see Sect. \ref{ss:sum-by-parts}).
 It can be shown that if $\gam<0<\gap$, the $n$-the derivative of $q/(q+\psi(\xi))$, each integrand in the formulas for the Wiener-Hopf factors, hence,  the price are of the order of  $O(|q|^{-n})$ as $q=\sg+iu\to \infty$ along the line of integration $\{\Re q=\sg\}$. Hence, applying the summation-by-parts procedure 3 times, one can
 reduce to the series which decays fairly fast, hence, the truncated sum with several hundreds of terms can satisfy a moderately small error tolerance. However, as in the case when the Gaver-Wynn Rho acceleration method is applied, the Wiener-Hopf factors have to be calculated for each $q$ in the truncated sum. Since $\psi^0(\eta)/(q-i\mu\eta)\to 0$ as $(q,\eta)\to \infty$
 ($q$ along the line of integration, and $\eta$ in the intersection of the half-plane $\{\mu\Im\eta>0\}$ and a domain of analyticity), the sinh-deformed contours for an efficient evaluation of the Wiener-Hopf factors are easy to construct.

  \section{Figures and tables}\label{ss:figures}
  
  \begin{figure}
\scalebox{0.75}
{\includegraphics{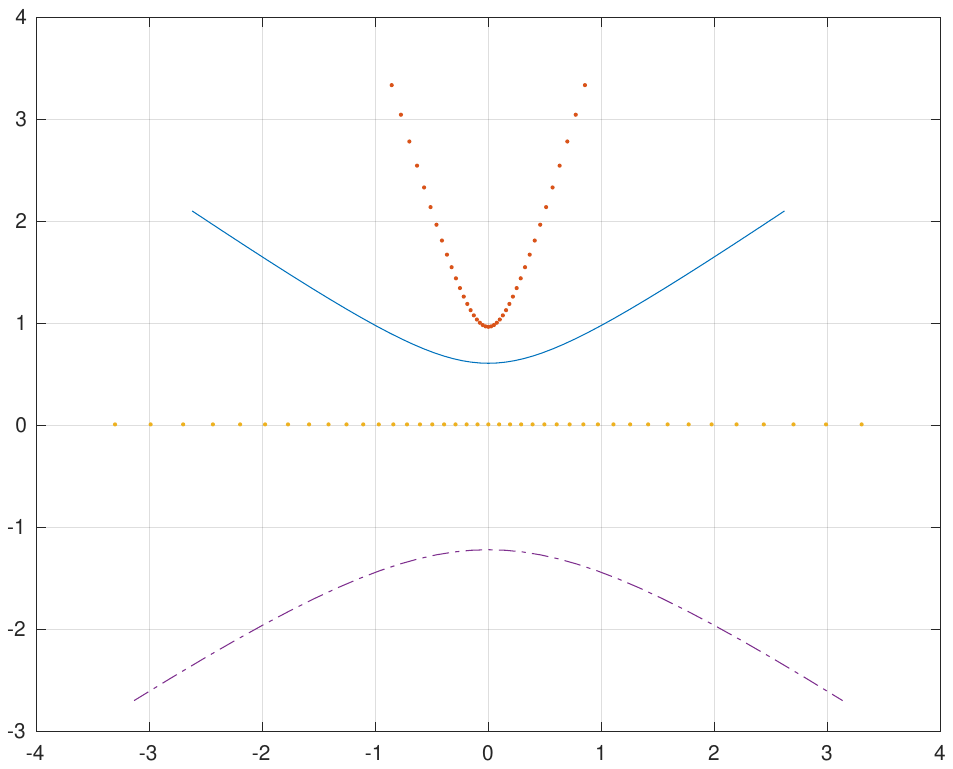}}
\caption{\small Example of curves $\cL^+$ (solid line) and $\cL^-$ (dash-dots). Example with $\lp=1, \lm=-2, \nu=1.2$. Dots: boundaries of the domain of analyticity
around $\cL^+$ used to derive the bound for the discretization error of the infinite trapezoid rule in the $y$-coordinate: dotted lines become
straight lines $\{\Im y=\pm d\}$.}
\label{fig:TwoGraphsandStrip}
\end{figure}
  \begin{figure}
\scalebox{0.75}
{\includegraphics{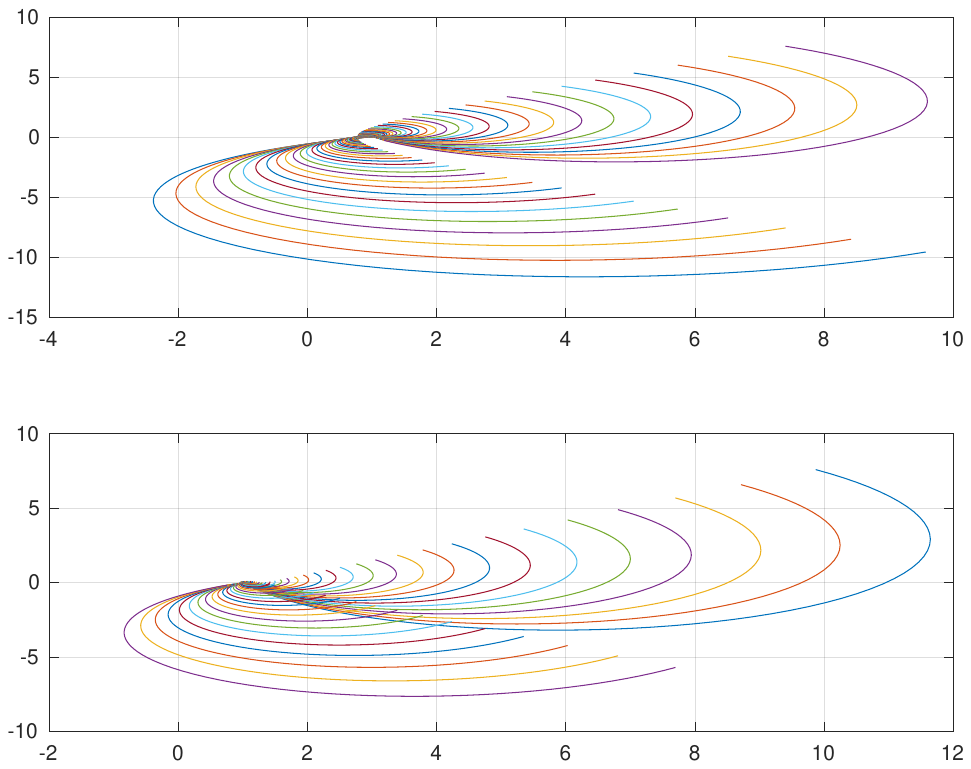}}
\caption{\small Plots of curves $\eta\mapsto 1+\psi(\eta)/q$, for $q$ in the SINH-Laplace inversion and $\eta$ on the contours $\cL^\pm$  (upper and lower panels)
in the numerical example with $\nu=1.2$.}
\label{Qphipm}
\end{figure}

\begin{table}
\caption{\small Joint cpdf $F(T,a_1,a_2):=\bQ[X_T\le a_1, \barX_T\le a_2\ |\ X_0=\barX_0=0]$. KoBoL 
close to NIG, with an almost symmetric jump density, and no ``drift": $m_2=0.1$, $\nu=1.2, \lm=-2, \lp=1$.
 }
 {\tiny
\begin{tabular}{c|c|cc|c|cc}
\hline\hline
$T=0.05$ & & & & & \\\hline
$a_2/a_1$ & -0.075 & -0.05 & -0.025 & 0 & 0.025 \\\hline
0.025 & 0.0426345508873718 &0.0758341778428274 & 0.176479681837557 & 0.493783805726552 & 0.76399258839732\\
0.05 & 0.0446956827465834 & 0.0789187973252002 & 0.181782048757841 & 0.506036792145469 & 0.825492125538671\\
0.075 &  0.0450873920315921 & 0.079458899594002 & 0.182586511426248 & 0.507408036688672 & 0.828608126596909 \\
0.1 &  0.0452106318743183 & 0.0796204107687271 & 0.182808929783218 & 0.507738655589658 & 0.829169593624407\\
0.175 & 0.0452978231524441 & 0.0797292390171655 & 0.182948969868149 & 0.507926759921863 & 0.829439308709987\\\hline

$T=0.25$ & & & & & \\\hline
0.025 & 0.163806126424503 & 0.222533168794254 & 0.292815888435677 & 0.358988211793687& 0.393398675917049 \\
0.05 & 0.197831466772809 & 0.270940241301468 & 0.364113238782974 & 0.465880513837506 & 0.552855276262823\\
0.075 & 0.209526961250121 & 0.287054894532268 & 0.387393027996113 & 0.501316508355731 & 0.609524332900865\\
0.1 & 0.214159056436717 & 0.293191765138545 & 0.395866562093269 & 0.513635238184172 & 0.628571055479703\\
0.175 & 0.217748710666063 & 0.297727492839728 & 0.401770665632438 & 0.521618850122037 & 0.639907339969623\\\hline
$T=1$ & & & & & \\\hline
0.025 &  0.178941818286114 & 0.190289038594875 & 0.199647908813292 & 0.206347351121367 & 0.209437388152747\\
0.05 &  0.260426736227358 & 0.280223680907225 & 0.29788417893014 & 0.312420272173741 & 0.322811896989456\\
0.075 & 0.313477022993733 & 0.340285459289499 & 0.365414405617852 & 0.387710627803938 & 0.405984471788573\\
0.1 & 0.348622321432066 &0.380779386234454 &  0.411911217215892 & 0.440836201412271 & 0.466304790550708\\
0.175 & 0.397364133265805 & 0.437693401916372 & 0.478455631551985 & 0.518663398654916 & 0.55725449371475\\\hline

$T=5$ & & & & & \\\hline
0.025 &  0.111436716966636 & 0.112239673751285 & 0.112868052194392 &0.113305936381633 & 0.113508446236642\\
0.05 &  0.260426736227358 & 0.280223680907225 & 0.29788417893014 & 0.312420272173741 & 0.322811896989456\\
0.075 & 0.313477022993733 & 0.340285459289499 & 0.365414405617852 & 0.387710627803938 & 0.405984471788573\\
0.1 & 0.348622321432066 &0.380779386234454 &  0.411911217215892 & 0.440836201412271 & 0.466304790550708\\
0.175 & 0.368564902845242 & 0.374400775302313 & 0.379792301029147 & 0.384713109733035 & 0.389137993602727\\\hline

$T=15$ & & & & & \\\hline

0.025 &  0.083599231183863 & 0.083725522194071 & 0.0838241629685378 & 0.0838929695457668 & 0.0839249287233805\\
0.05 & 0.130217710987261 & 0.130456839782095 & 0.130654399607263 &  0.130808705570046 & 0.13091634106018\\
0.075 & 0.169363038877019 & 0.169728032397852 & 0.170040043384744 & 0.170297815998657 & 0.170499151715123\\
0.1 & 0.204270598983103 & 0.204774983963964 & 0.205216260776888 & 0.205593481299844 & 0.205905127535884\\
0.175 & 0.293472724302235 & 0.294468206269081 & 0.295374834640356 & 0.296192060853885 & 0.296919211526691 \\\hline
\end{tabular}
}

\label{table:cpdf2}
 \end{table}

 \end{document}